\documentclass[a4paper,12pt,notitlepage]{amsart}
\usepackage{amsfonts,amsmath,amscd,amsthm,amssymb,graphicx}
\usepackage{booktabs,mathtools}
\usepackage{tikz,url}
\newtheorem{theorem}{Theorem}[section]
\newtheorem*{theorem*}{Theorem}
\newtheorem{proposition}[theorem]{Proposition}
\newtheorem{lemma}[theorem]{Lemma}
\newtheorem{corollary}[theorem]{Corollary}
\newtheorem{example}[theorem]{Example}
\newtheorem{remark}[theorem]{Remark}
\theoremstyle{definition}
\newtheorem{definition}[theorem]{Definition}

\newtheorem*{acknowledgement*}{Acknowledgements}

\DeclareMathOperator{\Ind}{Ind}
\DeclareMathOperator{\supp}{supp}
\newcommand{\ME}{\mathcal{ME}}
\newcommand{\cF}{\mathcal{F}}
\newcommand{\cA}{\mathcal{A}}
\makeatletter
\newcommand*{\rom}[1]{\expandafter\@slowromancap\romannumeral #1@}
\makeatother

\begin{document}

\title{On the Morse Ensemble Polynomial of Simplicial Complexes}

\author{Chong Zheng}

\address{
Department of Architecture,
Shibaura Institute of Technology,
Tokyo, Japan
}

\address{
Research Institute for Science and Engineering,
Waseda University,
Tokyo, Japan
}

\email{\url{c_zheng@aoni.waseda.jp}}


\begin{abstract}
We introduce the \emph{Morse ensemble polynomial}
$\ME_K(z_0,\ldots,z_d)$ of a finite simplicial complex $K$,
defined as the generating function
$\ME_K = \sum_M \prod_i z_i^{c_i(M)}$
over all acyclic matchings $M$ on the face poset of $K$, where
$c_i(M)$ counts critical $i$-simplices.
This polynomial records the complete critical-vector distribution over all
acyclic matchings, equivalently over all discrete gradient vector fields arising
from discrete Morse functions on $K$, and is an isomorphism invariant of
simplicial complexes.

In dimension one, this invariant recovers the Chari--Joswig graph formula
for the $f$-vector of the discrete Morse complex in a two-variable
Morse-vector normalization:
$\ME_G=z_1^{m-n}\det(z_0z_1\,I_n+L_G)$ for a connected graph $G$.
The main new contributions are higher-dimensional and structural.
First, we prove a Top-Face Recursion for adding a top-dimensional simplex,
with a non-liftable correction term $F(K,\sigma,\tau)$.
The vanishing and leading obstruction of this correction term are controlled
by the top incidence graph: an incidence-separation criterion detects exactly
when $F=0$, a leading obstruction layer is governed by shortest obstruction
paths, and a tree-like (incidence-forest) regime of the top incidence structure
gives a correction-free higher-dimensional recursion, including stacked balls
as a concrete class.  
Second, we introduce the independence ME polynomial
$\Phi(G):=\ME_{\mathrm{Ind}(G)}$, a graph invariant which strictly refines
the graph-level Morse ensemble, separates examples not distinguished by
$T_G$ and $I(G;t)$, and records collapse-level information of
$\mathrm{Ind}(G)$ through coefficients such as $[z_0]\Phi(G)$.
\end{abstract}

\maketitle

\section{Introduction}
\label{sec:intro}

\textit{Discrete Morse theory}, introduced by Forman~\cite{Forman1998},
provides a combinatorial framework for simplifying cell complexes while
preserving their homotopy type.  In this framework, a discrete \textit{gradient
vector field} on a finite cell complex can be described as an \textit{acyclic
matching} on its face poset, and the unmatched cells play the role of
\textit{critical} cells.  Much of the classical theory, as well as many algorithmic
applications, focuses on finding matchings with as few critical cells as
possible.  This optimisation problem is difficult in general: finding an
optimal discrete Morse matching is NP-hard~\cite{JoswigPfetsch2006,
BauerRathod2019}.

Beyond the search for a single optimal matching, the collection of all
acyclic matchings on a fixed complex has also been studied from several
viewpoints.  For example, the Chari--Joswig complex~\cite{ChariJoswig2005}
organises acyclic matchings into a simplicial complex, and our previous works
have investigated \textit{birth--death phenomena} among discrete Morse
functions~\cite{Zheng2024a,Zheng2024b,Zheng2025}.  These approaches emphasise the
structure of the space of discrete Morse functions, as well as the ways in
which such functions are connected or transformed into one another.

The aim of the present paper is different and complementary.  Our
contribution is not the mere consideration of all acyclic matchings, but
rather the introduction and study of a compact enumerative invariant that
records their critical-vector distribution.  More precisely, for a finite
simplicial complex $K$, we associate to the set of all acyclic matchings on
the face poset $P(K)$ a generating polynomial whose monomials encode the
numbers of critical simplices in each dimension.  This polynomial packages
the ensemble of discrete Morse matchings into a single algebraic object,
which we study as a combinatorial invariant of $K$.

Concretely, let $K$ be a finite simplicial complex of dimension $d$.  Each
acyclic matching $M$ on the face poset $P(K)$ determines a \emph{Morse
vector}
$
  c(M)=(c_0(M),\ldots,c_d(M)),
$
where $c_i(M)$ denotes the number of unmatched, or critical,
$i$-simplices. 
\begin{definition}
\label{def:ME}
The \emph{Morse ensemble polynomial} of $K$ is
\[
  \ME_K(z_0,\ldots,z_d)
  =\sum_{M\in\cA(K)}\prod_{i=0}^d z_i^{c_i(M)},
\]
where $\cA(K)$ denotes the set of all acyclic matchings on $P(K)$.
\end{definition}

The polynomial $\ME_K$ is finite, has non-negative integer coefficients,
and is an isomorphism invariant of simplicial complexes.  It refines the usual
optimisation problem in discrete Morse theory by recording, with multiplicity,
the full critical-vector distribution over all acyclic matchings.

The graph case anchors the invariant and supplies the spectral dictionary.  For
a connected graph $G$ with $n$ vertices, $m$ edges, and Laplacian $L_G$, the
Morse ensemble polynomial admits the closed Laplacian form
\[
  \ME_G(z_0,z_1)
  = z_1^{m-n}\det(z_0z_1 I_n+L_G)
  = z_1^{m-n}\prod_{i=1}^n (z_0z_1+\lambda_i),
\]
where $0=\lambda_1\leq\cdots\leq\lambda_n$ are the Laplacian eigenvalues.
Since graph Morse vectors satisfy $c_1=c_0+m-n$, the two variables collapse to
the single spectral parameter $q=z_0z_1$, so that $\ME_G$ is controlled entirely
by the Laplacian spectrum.  This identity is the Morse-vector renormalization of the Chari--Joswig graph
enumeration; Remarks~\ref{rem:CT-graph} and~\ref{cor:morse-complex-graph}
make the relation and attribution precise.  It gives a dictionary between
matchings and spectra: $\ME_G$ is a complete Laplacian-spectral invariant
(Theorem~\ref{cor:cospectral}), its coefficients are both spectral and
rooted-forest counts (Corollary~\ref{cor:sym-poly}), and its perfect-matching
coefficient recovers the Kirchhoff Matrix--Tree count
(Theorem~\ref{thm:perfect}).  This one-parameter spectral collapse is special
to dimension one; in dimension at least two the critical-vector distribution is
genuinely multivariate, and new phenomena appear.

The main new contributions are the following.

\medskip
\noindent\textbf{(I) The Top-Face Recursion and its correction term}
(Theorem~\ref{thm:top-face}).
Let $\sigma$ be a $d$-simplex whose boundary is contained in $K$, and set
$K'=K\cup\{\sigma\}$. 
Then,
\[
  \ME_{K\cup\{\sigma\}}
  \;=\; z_d\cdot\ME_K
  \;+\; \sum_{\tau\prec\sigma} L_{\sigma,\tau},
\]
where $L_{\sigma,\tau}$ counts acyclic matchings
on $P(K')\setminus\{\sigma,\tau\}$ whose lift by the pair
$(\tau,\sigma)$ remains acyclic.
Equivalently,
\[
  L_{\sigma,\tau}
  =\ME_{P(K')\setminus\{\sigma,\tau\}}-F(K,\sigma,\tau).
\]
Although a closed form for $F$ is open in general
(Problem~\ref{prob:nonliftable-matchings}), its structure is accessible:
Lemma~\ref{thm:incidence-separation} gives an exact vanishing criterion and
Proposition~\ref{prop:leading-obstruction} identifies its leading layer.  The
incidence-forest regime yields a correction-free recursion, in particular for
stacked balls (Theorem~\ref{prop:leaf-order},
Corollary~\ref{cor:stacked-balls}), while pseudomanifold attachments admit a
deterministic dual-graph liftability test
(Proposition~\ref{prop:forest-routing}).  The graph bridge recursion is the
one-dimensional shadow of this higher-dimensional mechanism; the analogy with
Tutte deletion-contraction is explained in Remark~\ref{rem:V-paths}.

\medskip
\noindent\textbf{(II) The independence ME polynomial as a fine graph invariant}
(Theorems~\ref{thm:main}, \ref{thm:phi-cospectral},
\ref{thm:phi-determines-me}).
The \textit{independence ME polynomial}
$\Phi(G):=\ME_{\Ind(G)}$ fits into a strict hierarchy with the
graph-level Morse ensemble $\ME_G$: the invariant $\Phi(G)$ determines
$\ME_G$, and the determination is strict. More precisely:
\begin{theorem*}[Hierarchy; Theorems~\ref{thm:main},
\ref{thm:phi-cospectral}, \ref{thm:phi-determines-me}]
For finite graphs:
\begin{enumerate}
\item[\normalfont(i)] there exist non-isomorphic graphs with
  $T_{G_1}=T_{G_2}$ and $I(G_1;t)=I(G_2;t)$ but
  $\Phi(G_1)\neq\Phi(G_2)$, even when
  $\Ind(G_1)\simeq\Ind(G_2)$;
\item[\normalfont(ii)] there exist Laplacian-cospectral graphs
  (so $\ME_{G_1}=\ME_{G_2}$) with
  $\Phi(G_1)\neq\Phi(G_2)$;
\item[\normalfont(iii)] conversely, $\Phi(G)$ determines $\ME_G$;
  hence $\Phi$ is a strict refinement of the graph-level Morse
  ensemble invariant;
\end{enumerate}
\end{theorem*}
The coefficient $[z_0]\Phi(G)$ gives an additional collapse-level
interpretation: it counts collapsing matchings of $\Ind(G)$, and is positive
precisely when $\Ind(G)$ is collapsible.

\medskip
The paper is organised as follows.
Section~\ref{sec:forest} introduces $\ME_K$, relates it to the
Chari--Joswig complex, develops the Forest Expansion and Compositional Lemma,
and derives exact formulas for paths and cycles.
Section~\ref{sec:structural} establishes the Bridge Recursion.
Section~\ref{sec:laplacian} records the Chari--Joswig graph formula in Morse-ensemble form and develops its spectral consequences.
Section~\ref{sec:separation} establishes the Top-Face
Recursion, the incidence-separation criterion controlling the correction
$F$, and the leading obstruction layer; it also records the Betti
coefficient that counts perfect matchings.
Section~\ref{sec:independence} proves the hierarchy theorem for the independence ME polynomial.
Section~\ref{sec:rigidity} collects open problems.

\section{The Forest Expansion}
\label{sec:forest}
We recall the basic notation for acyclic matchings and discrete Morse theory. 
The \emph{face poset} $P(K)$ of a finite simplicial complex $K$ is the
partially ordered set of all simplices of $K$, ordered by inclusion.
Its covering relations are precisely the pairs $\sigma\prec\tau$ with
$\sigma\subset\tau$ and $\dim\tau=\dim\sigma+1$.

An \emph{acyclic matching} on $P(K)$ is a collection $M$ of covering pairs
$(\sigma,\tau)$ such that each simplex of $K$ appears in at most one pair
and such that the Hasse diagram, after reversing the matched edges,
contains no directed cycle ~\cite{Forman1998}.  Equivalently, $M$ admits no
\emph{closed gradient path}: a sequence
$\sigma_0,\tau_0,\sigma_1,\tau_1,\ldots,\sigma_r=\sigma_0$ ($r\ge 1$) with
$(\sigma_i,\tau_i)\in M$, $\sigma_{i+1}\prec\tau_i$, and
$\sigma_{i+1}\neq\sigma_i$; such a closed gradient path is precisely a directed
cycle in the reversed Hasse diagram, and we use the two descriptions
interchangeably.  The simplices not appearing in any pair of
$M$ are called \emph{critical}.  We write $c_i(M)$ for the number of
critical $i$-simplices.  Thus every acyclic matching determines a
\emph{Morse vector}
\[
  c(M)=(c_0(M),\ldots,c_d(M)).
\]
By Forman's theorem~\cite{Forman1998}, if $K$ has dimension $d$, then
$K$ is homotopy equivalent to a CW complex with exactly $c_i(M)$ cells
of dimension $i$.  In particular, the Morse vector satisfies the \textit{Morse
inequalities}
$
  c_i(M)\geq \beta_i(K).
$
A matching for which $c_i(M)=\beta_i(K)$ for all $i$ is called
\emph{perfect}; such matchings need not exist.

The set of all acyclic matchings on $P(K)$ will be denoted by $\cA(K)$.
As defined in the introduction, the Morse ensemble polynomial is
\[
  \ME_K(z_0,\ldots,z_d)
  =
  \sum_{M\in\cA(K)}
  \prod_{i=0}^d z_i^{c_i(M)},
\]
where $\cA(K)$ is the set of all acyclic matchings on $P(K)$.  The
polynomial $\ME_K$ records the Morse vectors of all acyclic matchings,
with multiplicities.
The set $\cA(K)$ also has a natural simplicial-complex structure.

\begin{proposition}[Collapsibility coefficient]
\label{prop:collapsibility-coefficient}
Let $K$ be a finite connected simplicial complex.  The coefficient
$[z_0]\ME_K$ counts acyclic matchings with exactly one critical vertex and no
other critical simplices.  In particular,
\[
  [z_0]\ME_K>0
  \quad\Longleftrightarrow\quad
  K \text{ is collapsible}.
\]
\end{proposition}

\begin{proof}
This is Forman's collapsibility criterion~\cite{Forman1998}: a discrete Morse
matching with one critical $0$-simplex and no other critical simplices is
equivalent to a collapse to a point.  The coefficient $[z_0]\ME_K$ counts
exactly such matchings.
\end{proof}

\subsection{Relation with the Chari--Joswig complex}
\label{sec:chari-joswig}

Following Kozlov~\cite{Kozlov1999}, Chari and Joswig~\cite{ChariJoswig2005}
showed that acyclic matchings on $P(K)$ form the simplices of the
\emph{Chari--Joswig complex} $\mathfrak{M}(K)$.  Equivalently, $\ME_K$ is a
critical-vector refinement of the face enumerator of
$\mathfrak{M}(K)$.  We use the convention
\[
  f_X(q)=\sum_{\emptyset\neq F\in X} q^{|F|-1}
\]
for the ordinary face enumerator of a simplicial complex $X$.  If all variables
are specialised to a single variable $t$, then
\[
  \ME_K(t,\ldots,t)=t^{|K|}+t^{|K|-2}f_{\mathfrak M(K)}(t^{-2}),
\]
so $\ME_K$ recovers the $f$-vector of $\mathfrak M(K)$ after specialisation,
while distinguishing matchings of the same cardinality with different critical
vectors.  Equivalently, in terms of the multigraded Stanley--Reisner face
enumerator $F_{\mathfrak M(K)}$ whose vertex $(\sigma^i,\tau^{i+1})$ has
multidegree $e_i+e_{i+1}$,
\begin{equation}
\label{eq:SR-bridge}
  \ME_K(z_0,\ldots,z_d) =
  \left(\prod_{i=0}^d z_i^{f_i(K)}\right)
  F_{\mathfrak M(K)}\bigl(u_{(\sigma^i,\tau^{i+1})}=(z_i z_{i+1})^{-1}\bigr).
\end{equation}
In the collapsible case, $[z_0]\ME_K$ counts the facets of
$\mathfrak M(K)$ corresponding to matchings with a single critical vertex;
more generally, the Betti coefficient considered in
Subsection~\ref{subsec:perfect-coefficients} counts the faces of
$\mathfrak M(K)$ whose critical vector is the Betti vector.  For graphs,
Remark~\ref{cor:morse-complex-graph} recalls the specialised face enumerator
explicitly in terms of the nonzero Laplacian eigenvalues.

We next record a basic multiplicativity property.

\begin{proposition}[Multiplicativity]
\label{prop:product}
For disjoint simplicial complexes $K_1$ and $K_2$,
\[
  \ME_{K_1\sqcup K_2}=\ME_{K_1}\cdot\ME_{K_2}.
\]
\end{proposition}

\begin{proof}
An acyclic matching on $P(K_1\sqcup K_2)$ is uniquely a pair of acyclic
matchings $(M_1,M_2)$ on $P(K_1)$ and $P(K_2)$, respectively.  Moreover,
the critical-cell counts add:
\[
  c_i(M_1\sqcup M_2)=c_i(M_1)+c_i(M_2).
\]
The claimed identity follows immediately from the definition of
$\ME_K$.
\end{proof}

We next specialize to graphs.  Let $G=(V,E)$ be a connected graph with
$|V|=n$ and $|E|=m$, viewed as a one-dimensional simplicial complex.
Then $P(G)$ has elements $V\cup E$, ordered by $v<e$ whenever the
vertex $v$ is incident to the edge $e$.  Hence an acyclic matching on
$P(G)$ is a collection of vertex--edge pairs $(v,e)$, with $v\in e$,
satisfying the matching and acyclicity conditions above. 

For graphs, the Chari--Joswig complex admits a concrete interpretation.
It coincides with the complex of rooted spanning forests.  Equivalently,
an acyclic matching on $P(G)$ can be described by choosing a forest
support $F\subseteq G$ together with a root in each connected component
of $F$.  The bijection underlying this description is implicit in
Kozlov~\cite{Kozlov1999} and Chari--Joswig~\cite{ChariJoswig2005}.
For later use in the proof of the Laplacian formula, we state it below
as an explicit counting formula.

The following is a fundamental fact in discrete Morse theory.
\begin{lemma}[Forest Support~{\cite{Chari2000}}]
\label{lem:forest-support}
The edges appearing in any acyclic matching on $P(G)$ form a
forest in $G$.
\end{lemma}

For a forest $F\subseteq E$, let $a(F)$ denote the number of
acyclic matchings whose matched edges are exactly $F$.
Since the unmatched vertices number $n-|F|$ and unmatched edges
number $m-|F|$, summing over all possible forests gives the following equality.

\begin{proposition}
\label{prop:forest-sum}
$$\ME_G(z_0,z_1)
 =\sum_{F\in\cF(G)} a(F)\,z_0^{n-|F|}\,z_1^{m-|F|},$$
where $\cF(G)$ denotes the set of all spanning forests of $G$.
\end{proposition}

Proposition~\ref{prop:forest-sum} reduces computing $\ME_G$
to counting valid orientations of each forest $F$.
A \emph{valid orientation} of $F$ directs each edge such that
every vertex receives at most one incoming edge, equivalently,
each tree component $T$ of $F$ is rooted at its unique
\emph{escape vertex} with all edges directed away from it.
The following lemma gives the exact count.

\begin{lemma}[Compositional Lemma]
\label{lem:composition}
Let $F$ be a forest with tree components $T_1,\ldots,T_{c(F)}$.
Then,
\[
  a(F)=\prod_{i=1}^{c(F)}|V(T_i)|.
\]
In particular, $a(T)=|V(T)|$ for any tree $T$.
\end{lemma}

\begin{proof}
We show $a(T)=\ell$ for a tree $T$ on $\ell$ vertices.
For each $v\in V(T)$, root $T$ at $v$ and orient each edge
$\{u,\mathrm{parent}(u)\}$ toward $u$ (away from $v$). Every
non-root vertex receives exactly one incoming edge; the root
receives none. This defines a valid orientation $\omega_v$ with
\emph{escape vertex} $v$.

We claim $v\mapsto\omega_v$ is a bijection. Each valid orientation
has exactly one escape vertex: since $\ell-1$ edges each point into
one distinct vertex, exactly one vertex receives no incoming edge.
Given the escape vertex $v$, the orientation is uniquely determined
(root $T$ at $v$, direct each edge toward the child). Two distinct
choices of $v$ yield different orientations, since the edge on the
unique path between them reverses direction. Hence $a(T)=\ell$.

For a forest, orientations of distinct components are independent,
so $$a(F)=\prod a(T_i)=\prod|V(T_i)|.$$
\end{proof}

Note that the identity $a(T)=|V(T)|$ is shape-independent.
For example,  the path graph $P_\ell$
and the star $K_{1,\ell-1}$ both have $a=\ell$, despite having
entirely different structures.

Combining Proposition~\ref{prop:forest-sum} with
Lemma~\ref{lem:composition} gives an explicit formula for the
Morse ensemble on graphs.
\begin{theorem}[Forest Expansion]
\label{thm:forest-expansion}
\[
  \ME_G(z_0,z_1)
  =\sum_{F\in\cF(G)}\!\Bigl(\prod_{i=1}^{c(F)}|V(T_i(F))|\Bigr)
   z_0^{n-|F|}\,z_1^{m-|F|},
\]
where $c(F)$ denotes the number of connected components of $F$
and $T_1(F),\ldots,T_{c(F)}(F)$ are its tree components.
\end{theorem}

\begin{proof}
This follows immediately from the expansion over forest supports and
Lemma~\ref{lem:composition}.
\end{proof}

\subsection{Exact Formulas for Paths and Cycles}
\label{sec:exact}
As a first application of the forest expansion, 
we compute $\ME$ for paths and cycles, recovering Fibonacci and Lucas
numbers as special cases by evaluating at $z_0=z_1=1$.

Let $P_n$ denote the path on $n$ vertices. 
Since every subgraph of
$P_n$ is a forest, a $k$-component spanning
forest corresponds to a composition $$\ell_1+\cdots+\ell_k = n,$$
where $\ell_i \geq 1$, with compositional weight $\prod\ell_i$ and Morse
vector $(c_0, c_1) = (k, k-1)$.

\begin{theorem}[Path Formula]
\label{thm:path-formula}
\[
  \ME_{P_n}(z_0,z_1) = \sum_{k=1}^{n}\binom{n+k-1}{2k-1}\,z_0^k z_1^{k-1}.
\]
\end{theorem}

\begin{proof}
The coefficient of $z_0^k z_1^{k-1}$ equals
$$\sum_{\ell_1+\cdots+\ell_k=n,\,\ell_i\geq 1}\prod\ell_i,$$
which evaluates to $$[x^n]\bigl(x/(1-x)^2\bigr)^k = \binom{n+k-1}{2k-1}$$
by the generating function $\sum_{\ell\geq 1}\ell x^\ell = x/(1-x)^2$.
\end{proof}

Next let $C_n$ denote the cyclic graph on $n$ vertices. A spanning forest of
$C_n$ with $k$ connected components is obtained by deleting $k$ edges
from the cycle,
leaving $k$ path arcs of lengths summing to $n$, with Morse vector
$(k, k)$.

\begin{theorem}[Cycle Formula]
\label{thm:cycle-formula}
\[
  \ME_{C_n}(z_0,z_1) = \sum_{k=1}^{n}\tfrac{n}{k}\binom{n+k-1}{2k-1}(z_0 z_1)^k.
\]
\end{theorem}

\begin{proof}
Each $k$-component forest of $C_n$ arises from a starting position
$s \in \mathbb{Z}/n\mathbb{Z}$ and a linear composition $L$ of $n$
into $k$ parts; the pair $(s,L)$ is $k$-to-$1$ over forests. Summing
the Compositional weight $\prod\ell_i$ over all pairs and dividing by
$k$ yields $\frac{n}{k}\binom{n+k-1}{2k-1}.$
\end{proof}

Evaluating at $z_0=z_1=1$ recovers classical combinatorial identities.

\begin{theorem}[Spectral Fibonacci--Lucas identities]
\label{thm:fibonacci}
\label{thm:lucas}
For $n\geq 1$,

$$|\cA(P_n)| = \ME_{P_n}(1,1) = F_{2n},$$
and, for $n\geq 3$,
$$|\cA(C_n)| = \ME_{C_n}(1,1) = L_{2n} - 2,$$
where $F_m$ and $L_m$ are the $m$-th Fibonacci and Lucas numbers.

\end{theorem}

\begin{proof}
For paths, the sum $\sum_{k=1}^n\binom{n+k-1}{2k-1} = F_{2n}$
follows from the Fibonacci identity
$$B_m := \sum_{r=0}^{\lfloor m/2\rfloor}\binom{m-r}{r} = F_{m+1},$$
which is proved by a standard Pascal recursion.
The Lucas identity then follows from $L_{2n} = F_{2n+1} + F_{2n-1}$
together with the analogous sum for cycles.
\end{proof}
These also follow from the Laplacian Formula via
$|\cA(G)| = \prod_i(1+\lambda_i)$
and the known spectra of $P_n$ and $C_n$.
\section{The Bridge Recursion}
\label{sec:structural}

Next, we derive a deletion-type recursion for bridge edges,
analogous in spirit to the Tutte deletion-contraction formula.
For non-bridge edges, the same naive three-term formula is no longer correction-free.
Non-liftable matchings appear, and these are exactly the one-dimensional 
instance of the correction terms discussed later by the Top-Face Recursion.

\begin{theorem}[Bridge Recursion]
\label{thm:recursion}
Let $e=\{u,v\}$ be a bridge in $G$.  Then
\begin{equation}
\label{eq:3term}
\tag{$\star$}
  \ME_G = z_1\cdot\ME_{G\setminus e}
  + \ME_{P(G)_{e\to u}}
  + \ME_{P(G)_{e\to v}},
\end{equation}
where $P(G)_{e\to u}$ (resp.\ $P(G)_{e\to v}$) is the sub-poset
of $P(G)$ obtained by deleting the two elements $u,e$ (resp.\ $v,e$), together
with all covering relations incident to them.
\end{theorem}

\begin{proof}
We partition $\cA(G)$ according to the status of the edge $e$.
\begin{itemize}
\item \emph{$e$ critical.} Bijection with $\cA(G\setminus e)$;
      the edge $e$ contributes a critical $1$-cell, giving
      $z_1\cdot\ME_{G\setminus e}$.
\item \emph{$e$ matched with $u$.} The pair $(u,e)$ is fixed.
      Restriction gives a bijection with acyclic matchings on
      $P(G)_{e\to u}$: acyclicity is preserved because $e$ is
      a bridge, so $G\setminus e$ has no path from $v$ to $u$,
      preventing a directed cycle through $e$.
      This contributes $\ME_{P(G)_{e\to u}}$.
\item \emph{$e$ matched with $v$.} Symmetric, contributing
      $\ME_{P(G)_{e\to v}}$.
\end{itemize}
Adding the three disjoint cases gives the formula.
\qedhere
\end{proof}

\begin{remark}
\label{rem:non-bridge}
Theorem~\ref{thm:recursion} is the  case $d=1$ of the
Top-Face Recursion (Theorem~\ref{thm:top-face}).
The first term
$z_1\ME_{G\setminus e}$ plays the role of deletion.  The other two terms
play the role of contraction, but the operation is performed in the face
poset: fixing $(u,e)$ removes the matched vertex $u$ and the edge $e$,
rather than identifying $u$ and $v$ as in the usual graph contraction.
This is why the Morse ensemble recursion has two contraction-type terms,
$\ME_{P(G)_{e\to u}}$ and $\ME_{P(G)_{e\to v}}$, instead of one. 

 For a non-bridge edge, on the other hand,
formula~\eqref{eq:3term} may overcount via non-liftable matchings,
handled in dimension one by the correction term of the Top-Face
Recursion. For example,
for the cycle $C_n$, the overcount is exactly $2$, consistent with the
identity
\[
  |\cA(C_n)|=L_{2n}-2.
\]
\end{remark}

\section{The Laplacian Formula and Spectral Theory of $\ME_G$}
\label{sec:laplacian}

We now record the graph case in the normalization of the Morse ensemble
polynomial.  As discussed in Remark~\ref{rem:CT-graph}, the resulting formula
is equivalent to the Chari--Joswig formula for the $f$-vector of the graph
Morse complex.  It is nevertheless useful to derive it directly from the
Forest Expansion, since this makes the Morse-vector normalization explicit and
provides the spectral dictionary used later.  The key bridge is the
Matrix-Forest Theorem.
\begin{theorem}[Matrix-Forest Theorem~{\cite{Chebotarev1997}}]
\label{thm:MFT}
For any graph $G$ with $n$-vertex Laplacian $L_G$,
\[
  \det(\lambda I_n+L_G)
  =\sum_{F\in\cF(G)}\Bigl(\prod_{i=1}^{c(F)}|V(T_i(F))|\Bigr)
   \lambda^{c(F)}.
\]
\end{theorem}

Combining the Forest Expansion with the Matrix-Forest Theorem gives the graph
formula in Morse-ensemble form.
\begin{theorem}[Graph formula in Morse-ensemble form]
\label{thm:laplacian}
For any connected graph $G$ with $n$ vertices, $m$ edges, Laplacian
$L_G$, and eigenvalues $0=\lambda_1\leq\cdots\leq\lambda_n$:
\[
  \ME_G(z_0,z_1)
  =z_1^{m-n}\cdot\det(z_0 z_1\,I_n+L_G)
  =z_1^{m-n}\prod_{i=1}^{n}(z_0 z_1+\lambda_i).
\]
\end{theorem}

\begin{proof}
For any spanning forest $F$, the Euler characteristic for forests
($\text{vertices} - \text{edges} = \text{components}$) gives
$n-|F|=c(F)$, hence $m-|F|=m-n+c(F)$.
The Forest Expansion (Theorem~\ref{thm:forest-expansion}) can be rewritten as follows.
\begin{align*}
  \ME_G(z_0,z_1)
  &=\sum_{F\in\cF(G)}\!\Bigl(\prod_{i=1}^{c(F)}|V(T_i)|\Bigr)
    z_0^{n-|F|}\,z_1^{m-|F|}\\
  &=\sum_F\!\Bigl(\prod|V(T_i)|\Bigr)\,
    z_0^{c(F)}\,z_1^{m-n+c(F)}\\
  &=z_1^{m-n}\sum_F\!\Bigl(\prod|V(T_i)|\Bigr)(z_0 z_1)^{c(F)}.
\end{align*}
Setting $\lambda=z_0 z_1$, the inner sum
$\sum_F\bigl(\prod|V(T_i)|\bigr)\lambda^{c(F)}$
equals $\det(\lambda I_n+L_G)$ by the Matrix-Forest Theorem
(Theorem~\ref{thm:MFT}), completing the proof.
\end{proof}
\begin{remark}[Relation with Chari--Joswig and Contreras--Tawfeek]
\label{rem:CT-graph}
For a connected graph $G$, the
Chari--Joswig formula for the $f$-vector of the graph Morse complex
$\mathfrak M(G)$ gives
\[
  \sum_{k\ge 0} f_{k-1}(\mathfrak M(G))q^k
  =\prod_{i=2}^{n}(1+\lambda_i q),
\]
where $0=\lambda_1\leq\lambda_2\leq\cdots\leq\lambda_n$ are the Laplacian eigenvalues.
Contreras and Tawfeek~\cite{ContrerasTawfeek2024} also emphasise this
Laplacian interpretation in their study of discrete gradient vector fields and
combinatorial Laplacians.  In the present notation, the same enumeration
becomes Theorem~\ref{thm:laplacian}: a graph matching with $k$ matched pairs
has Morse vector $(n-k,m-k)$, so the change of variables $q=z_0z_1$ and the
normalising factor $z_1^{m-n}$ convert the Chari--Joswig $f$-polynomial into
$\ME_G(z_0,z_1)$.

Thus the graph formula serves as the one-dimensional base case and fixes the
normalization of the Morse ensemble polynomial.  The new direction of the
present paper begins beyond this one-parameter graph situation: for general
simplicial complexes the critical-vector distribution is genuinely
multivariate, and the Top-Face Recursion, the non-liftable correction term
$F(K,\sigma,\tau)$, the incidence-separation and leading-obstruction criteria,
and the independence ensemble $\Phi(G)=\ME_{\Ind(G)}$ are not consequences of
the graph Morse-complex $f$-vector formula.
\end{remark}

For trees ($m=n-1$), $z_1^{m-n}=z_1^{-1}$. Since $\lambda_1=0$
is always a Laplacian eigenvalue, $\det(z_0 z_1\,I+L_G)$ is
divisible by $z_0 z_1$ (hence by $z_1$), making $\ME_G$ a polynomial with nonnegative integer exponents.

Setting $z_0=z_1=1$ in Theorem~\ref{thm:laplacian} gives
the immediate spectral count
\begin{equation}
\label{eq:spectral-count}
|\cA(G)|=\det(I_n+L_G)=\prod_{i=1}^n(1+\lambda_i(G)).
\end{equation}

\begin{definition}[Laplacian-cospectral]
Two graphs $G_1$ and $G_2$ are \emph{Laplacian-cospectral} if their
Laplacian matrices have the same multiset of eigenvalues. 
\end{definition}

\begin{theorem}[$\ME_G$ is a complete Laplacian spectral invariant]
\label{cor:cospectral}
Let $G_1$ and $G_2$ be connected graphs. Then $$\ME_{G_1}=\ME_{G_2}$$
if and only if $G_1$ and $G_2$ are Laplacian-cospectral.
\end{theorem}

\begin{proof}
($\Leftarrow$) Laplacian-cospectral graphs share the same number of
vertices ($n$ = number of eigenvalues) and edges
($2m=\mathrm{tr}(L_G)$). The Laplacian Formula
$\ME_G=z_1^{m-n}\prod_i(z_0z_1+\lambda_i)$ then determines $\ME_G$
from the multiset $\{\lambda_i\}$ and $m-n$.

($\Rightarrow$) Equality $\ME_{G_1}=\ME_{G_2}$ as polynomials in
$z_0,z_1$ matches degrees, hence the same $n,m$, and matches
all coefficients, hence equality of all elementary symmetric
polynomials $e_{n-j}(\lambda_1,\ldots,\lambda_n)$ for
$j=0,\ldots,n$ (Corollary~\ref{cor:sym-poly}). Since the
elementary symmetric polynomials determine the multiset of roots,
the Laplacian spectra coincide.
\end{proof}

\begin{example}[Laplacian-cospectral pair]
\label{ex:cospectral}
Let $G_1$ and $G_2$ be the following connected graphs on $6$ vertices
and $7$ edges in Figure 1.
\begin{align*}
  E(G_1)&=\{01,02,03,05,14,23,45\},\\
  E(G_2)&=\{01,02,03,14,15,24,25\}.
\end{align*}

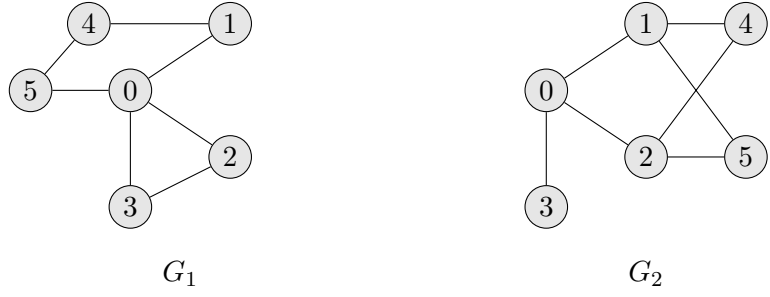
\begin{figure}[ht]
\centering
\begin{tikzpicture}[scale=1.1, every node/.style={circle,draw,fill=black!10,
  inner sep=2pt, minimum size=16pt, font=\small}]
\begin{scope}[xshift=0cm]
  \node (0) at (0,0)   {$0$};
  \node (1) at (1.2,0.8) {$1$};
  \node (2) at (1.2,-0.8){$2$};
  \node (3) at (0,-1.4) {$3$};
  \node (4) at (-0.5,0.8){$4$};
  \node (5) at (-1.2,0) {$5$};
  \draw (0)--(1) (0)--(2) (0)--(3) (0)--(5) (1)--(4) (2)--(3) (4)--(5);
  \node[draw=none,fill=none] at (0.6,-2.2) {$G_1$};
\end{scope}
\begin{scope}[xshift=5cm]
  \node (0) at (0,0)   {$0$};
  \node (1) at (1.2,0.8) {$1$};
  \node (2) at (1.2,-0.8){$2$};
  \node (3) at (0,-1.4) {$3$};
  \node (4) at (2.4,0.8){$4$};
  \node (5) at (2.4,-0.8){$5$};
  \draw (0)--(1) (0)--(2) (0)--(3) (1)--(4) (1)--(5) (2)--(4) (2)--(5);
  \node[draw=none,fill=none] at (1.2,-2.2) {$G_2$};
\end{scope}
\end{tikzpicture}
\caption{Laplacian-cospectral non-isomorphic graphs. $G_1$ (left)
contains a triangle $\{0,2,3\}$ (edges $02,03,23$ are all present),
while $G_2$ (right) is triangle-free. Both have the same $\ME$ polynomial
but different Tutte polynomials.}
\label{fig:cospectral}
\end{figure}

\begin{itemize}
\item [$(1)$]The graphs are non-isomorphic. $G_1$ has degree sequence $(4,2,2,2,2,2)$
while $G_2$ has degree sequence $(3,3,3,2,2,1)$.
\item[$(2)$] Both have Laplacian eigenvalues
$\bigl\{0,\,3-\sqrt{5},\,2,\,3,\,3,\,3+\sqrt{5}\bigr\}$,
hence by Theorem~\ref{cor:cospectral},
\[
  \ME_{G_1}=\ME_{G_2}
  = 72z_0z_1^2+192z_0^2z_1^3+176z_0^3z_1^4+73z_0^4z_1^5+14z_0^5z_1^6+z_0^6z_1^7.
\]
\item [$(3)$]The Tutte polynomials differ as follows.
\begin{align*}
  T_{G_1}&=x^5+2x^4+x^3y+2x^3+2x^2y+x^2+2xy+y^2,\\
  T_{G_2}&=x^5+2x^4+3x^3+3x^2y+x^2+xy^2+xy.
\end{align*}

\end{itemize}

This shows that $\ME_G$ does not determine $T_G$.
\end{example}

\subsection{Spectral coefficients dictionary}

\label{subsec:spectral-dictionary}
Recall that the \emph{$k$-th elementary symmetric polynomial}
of $n$ variables is
$$e_k(x_1,\ldots,x_n) = \sum_{1\leq i_1<\cdots<i_k\leq n}x_{i_1}\cdots x_{i_k},$$
such that $$\prod_{i=1}^n(t+x_i)=\sum_{k=0}^n e_{n-k}\,t^k.$$

\begin{corollary}[Coefficients as Symmetric Polynomials]
\label{cor:sym-poly}
For $j=0,1,\ldots,n$,
\begin{equation}
\label{eq:sym-coeffs}
  \bigl[z_0^{\,j}\,z_1^{m-n+j}\bigr]\,\ME_G
  = e_{n-j}(\lambda_1,\ldots,\lambda_n).
\end{equation}
In other words, the coefficients of $\ME_G$ along the antidiagonal
$\{z_0^j z_1^{m-n+j}\}_{j=0}^n$ are exactly the elementary symmetric
polynomials of the Laplacian eigenvalues.
\end{corollary}

\begin{proof}
Expanding the product in Theorem~\ref{thm:laplacian} as follows.
\begin{align*}
 \ME_G(z_0,z_1)
  &= z_1^{m-n}\prod_{i=1}^n(z_0z_1+\lambda_i)\\
  &= z_1^{m-n}\sum_{j=0}^n e_{n-j}(\lambda_1,\ldots,\lambda_n)\,(z_0z_1)^j\\
  &= \sum_{j=0}^n e_{n-j}\,z_0^{\,j}\,z_1^{m-n+j}.\qedhere
\end{align*}
\end{proof}

Equation~\eqref{eq:sym-coeffs} gives a dictionary between
the ME polynomial and the Laplacian spectrum as follows.
In particular,
\begin{align*}
  j=n:&\quad [z_0^nz_1^m]\ME_G = e_0 = 1
    &\text{(trivial matching)},\\
  j=n-1:&\quad [z_0^{n-1}z_1^{m-1}]\ME_G = e_1 = \textstyle\sum_i\lambda_i
    = \mathrm{tr}(L_G) = 2m
    &\text{(degree sum)},\\
  j=1:&\quad [z_0z_1^{m-n+1}]\ME_G = e_{n-1} = n\tau(G)
    &\text{(Matrix-Tree Theorem)}.
\end{align*}

The $j=n-1$ case has a direct combinatorial interpretation.
There are exactly $2m$ single-pair acyclic matchings (one pair
$(v,e)$ for each of the two orientations of each edge $e$), and
this equals $e_1=\mathrm{tr}(L_G)=2m$.
More generally, combining Corollary~\ref{cor:sym-poly} with the
Forest Expansion gives the rooted-forest dictionary
\begin{equation}
\label{eq:rooted-forest-dictionary}
  \bigl[z_0^{\,j}z_1^{m-n+j}\bigr]\ME_G
  =
  \sum_{\substack{F\in\mathcal F(G)\\ c(F)=j}}
  \prod_{T\in\pi_0(F)} |V(T)|.
\end{equation}
Thus the $j$-th coefficient counts rooted spanning forests with
$j$ components, where each tree component is rooted independently.
In particular, the case $j=1$ gives rooted spanning trees, and the
case $j=n$ is the trivial matching.

At the opposite end, the coefficient with $j=n-2$ counts two-pair
acyclic matchings. Since any two disjoint primitive pairs are acyclic,
we obtain the explicit degree formula
\begin{equation}
\label{eq:two-pair-coefficient}
  \bigl[z_0^{n-2}z_1^{m-2}\bigr]\ME_G
  =
  \binom{2m}{2}-m-
  \sum_{v\in V}\binom{\deg(v)}{2}.
\end{equation}
Indeed, one first chooses two oriented incidences among the $2m$ possible
vertex--edge pairs, and subtracts the choices sharing the same edge
or the same vertex. Equivalently, \eqref{eq:two-pair-coefficient}
gives a matching-theoretic interpretation of
$e_2(\lambda_1,\ldots,\lambda_n)$.
Thus $\ME_G$ simultaneously encodes the rooted spanning forest counts
by number of components, the near-trivial matching counts, and the
complete Laplacian spectrum.

A consequence of Corollary~\ref{cor:sym-poly}, which is also a straightforward result of discrete Morse theory, is that
$\ME_G$ has \emph{full support}. More precisely, along the antidiagonal forced by
the Euler relation, every monomial appears.

\begin{corollary}[Full support for connected graphs]
\label{thm:graph-support}
For a connected graph $G$ with $n$ vertices and $m$ edges, every
monomial $z_0^k\,z_1^{k+m-n}$ with $1\leq k\leq n$ appears in $\ME_G$
with positive coefficient. Equivalently,
\[
  \supp(\ME_G) \;=\; \bigl\{(k,\,k+m-n) : 1\leq k\leq n\bigr\}.
\]
\end{corollary}

\begin{proof}
By Corollary~\ref{cor:sym-poly}, the coefficient of
$z_0^k z_1^{k+m-n}$ in $\ME_G$ is
$e_{n-k}(\lambda_1,\ldots,\lambda_n)$.
Since $\lambda_1=0$ and all other eigenvalues are strictly positive
(connected graph), $e_{n-k}>0$ for $1\leq k\leq n$.
The Euler relation $\sum_i(-1)^i a_i = \chi(G) = n-m$
(applied to any Morse vector of $\ME_G$) forces
$a_1-a_0=m-n$, so the support lies along this antidiagonal.
\end{proof}

In higher dimensions,
$\ME_K$ need not have full support.

\begin{theorem}[Perfect Morse matchings; Kirchhoff via $\ME_G$]
\label{thm:perfect}
Let $G$ be a connected graph with $n$ vertices, $m$ edges, and
$\tau(G)$ spanning trees. The Betti numbers of $G$ are
$\beta_0=1$ and $\beta_1=m-n+1$.
The number of perfect acyclic matchings on $P(G)$ is
\[
  \bigl[z_0\,z_1^{m-n+1}\bigr]\,\ME_G
  = n\cdot\tau(G)
  = e_{n-1}(\lambda_1,\ldots,\lambda_n)
  = \lambda_2\lambda_3\cdots\lambda_n.
\]
In particular, perfect matchings always exist for connected graphs:
every spanning tree gives rise to $n$ such matchings, and the
identity $n\tau(G)=\lambda_2\cdots\lambda_n$ recovers the
Kirchhoff Matrix-Tree Theorem as a corollary of the Laplacian Formula.
\end{theorem}

\begin{proof}
A matching achieves $c_0=1$, $c_1=m-n+1$ if and only if the
matched edges form a spanning tree of $G$ (so that exactly one
vertex and $m-n+1$ edges remain critical). By the Compositional
Lemma (Lemma~\ref{lem:composition}), each spanning tree $T$
admits exactly $n$ valid orientations (one per choice of escape
vertex). Thus the count is $n\cdot\tau(G)$.
The spectral identity $n\tau(G)=e_{n-1}(\lambda_1,\ldots,\lambda_n)
=\lambda_2\cdots\lambda_n$ then follows from
Corollary~\ref{cor:sym-poly} with $j=1$ (using $\lambda_1=0$).
\end{proof}

The following remark makes explicit the recovery of the Chari--Joswig graph
formula in the notation of the present paper.

\begin{remark}[Recovering the Chari--Joswig graph formula]
\label{cor:morse-complex-graph}
Let $G$ be a connected graph on $n$ vertices with Laplacian eigenvalues
$0=\lambda_1<\lambda_2\leq\cdots\leq\lambda_n$, and let
$\mathfrak M(G)$ be the Chari--Joswig Morse complex.  Its $(k-1)$-faces are
precisely the acyclic matchings of $P(G)$ with $k$ matched pairs.  Such a
matching has critical vector $(n-k,m-k)$, and hence
Corollary~\ref{cor:sym-poly} gives
\[
  f_{k-1}(\mathfrak M(G))
  =e_k(\lambda_2,\ldots,\lambda_n).
\]
Equivalently,
\[
  W_G(q):=\sum_{M\in\mathcal A(P(G))} q^{|M|}
  =\sum_{k\ge0} f_{k-1}(\mathfrak M(G))q^k
  =\prod_{i=2}^{n}(1+\lambda_i q).
\]
Thus Theorem~\ref{thm:laplacian} recovers the Chari--Joswig $f$-vector formula
for the graph Morse complex after the change of variables $q=z_0z_1$ and the
normalization factor $z_1^{m-n}$.
\end{remark}

The graph formula also gives explicit formulas for graphs
built by Cartesian products, since the Laplacian of $G_1\square G_2$
is $L_1\otimes I_{n_2}+I_{n_1}\otimes L_2$ with eigenvalues
$\{\lambda_i+\mu_j\}$.

\begin{corollary}[Cartesian Products]
\label{thm:cartesian}
Let $G_1$ and $G_2$ be connected graphs with Laplacian eigenvalues
$\{\lambda_i\}_{i=1}^{n_1}$ and $\{\mu_j\}_{j=1}^{n_2}$.
Set $n=n_1n_2$, $m=m_1n_2+m_2n_1$.
Then
\[
  \ME_{G_1\square G_2}(z_0,z_1)
  = z_1^{m-n}\prod_{i=1}^{n_1}\prod_{j=1}^{n_2}
    (z_0z_1+\lambda_i+\mu_j).
\]
\end{corollary}

\begin{example}[Grid graphs and hypercubes]
\label{ex:cartesian}
The $m\times n$ grid $P_m\square P_n$ has eigenvalues
$\{(2-2\cos(k\pi/m))+(2-2\cos(\ell\pi/n))\}$.
The $n$-cube $Q_n=P_2^{\square n}$ has eigenvalues $\{2k\}$
with multiplicity $\binom{n}{k}$, giving
$|\cA(Q_n)|=\prod_{k=0}^n(1+2k)^{\binom{n}{k}}$, with values
$3, 45, 23625, 27348890625$ for $n=1,2,3,4$.
\end{example}

\subsection{Spectral identities and examples}

The Laplacian eigenvalues of $P_n$ and $C_n$ are
$\lambda_k^{P}=2-2\cos(k\pi/n)$ and $\lambda_k^{C}=2-2\cos(2\pi k/n)$.
Combining with $|\cA(G)|=\prod_i(1+\lambda_i)$ and
Theorem~\ref{thm:fibonacci} gives the spectral identities
as follows.
\begin{corollary}
$$F_{2n}=\prod_{k=1}^{n-1}(3-2\cos(k\pi/n))$$ 
and
$$L_{2n}-2=\prod_{k=1}^{n-1}(3-2\cos(2\pi k/n)).$$
\end{corollary}

For the complete graphs $K_n$, the eigenvalues are $0$ and $n$ (with multiplicity $n-1$),
recovering the classical count of rooted
spanning forests~\cite{Pitman2002}.
\begin{corollary}
$$|\cA(K_n)|=(n+1)^{n-1}.$$
\end{corollary} 

Table~\ref{tab:examples} summarises $\ME_G$ for fundamental graphs.

\begin{table}[ht]
\centering
\begin{tabular}{llcc}
\toprule
$G$ & $\ME_G(z_0,z_1)$ & $|\cA(G)|$ & Formula \\
\midrule
$P_3$ & $3z_0+4z_0^2z_1+z_0^3z_1^2$ & $8$ & $F_6$ \\
$P_4$ & $4z_0+10z_0^2z_1+6z_0^3z_1^2+z_0^4z_1^3$ & $21$ & $F_8$ \\
$C_3$ & $9z_0z_1+6z_0^2z_1^2+z_0^3z_1^3$ & $16$ & $L_6-2$ \\
$C_4$ & $16z_0z_1+20z_0^2z_1^2+8z_0^3z_1^3+z_0^4z_1^4$ & $45$ & $L_8-2$ \\
$K_{1,3}$ & $4z_0+9z_0^2z_1+6z_0^3z_1^2+z_0^4z_1^3$ & $20$ & \\
$K_4$ & $64z_0z_1^3+48z_0^2z_1^4+12z_0^3z_1^5+z_0^4z_1^6$
      & $125$ & $(1+4)^3$ \\
\bottomrule
\end{tabular}
\caption{Morse ensemble polynomials satisfying
$\ME_G=z_1^{m-n}\det(z_0z_1\,I+L_G)$.
Totals verified via $|\cA(G)|=\det(I+L_G)$.}
\label{tab:examples}
\end{table}
These examples illustrate the agreement between combinatorial counts
of acyclic matchings and spectral formulas derived from the Laplacian.

\subsection{Comparison with the Tutte polynomial}
\label{sec:tutte}

$\ME_G$ is a Laplacian-spectral invariant, while $T_G$ is a
cycle-matroid invariant. The following explicit witnesses show that
neither invariant determines the other.

\begin{proposition}[Incomparability]
\label{thm:incomparable}
$\ME_G$ and $T_G$ are incomparable invariants. That is, neither determines the other.
\end{proposition}

\begin{proof}
We give explicit witnesses in both directions. First, all trees on
$n$ vertices have the same Tutte polynomial $x^{n-1}$. However, the
path $P_4$ and the star $K_{1,3}$ have different Laplacian spectra,
so Theorem~\ref{thm:laplacian} gives $\ME_{P_4}\neq\ME_{K_{1,3}}$.
Thus $T_G$ does not determine $\ME_G$.

Conversely, the Laplacian-cospectral pair of Example~\ref{ex:cospectral}
has $\ME_{G_1}=\ME_{G_2}$ by Theorem~\ref{cor:cospectral}, but the
Tutte polynomials displayed there are different. Thus $\ME_G$ does not
determine $T_G$.
\end{proof}

\section{Top-Face Recursion and Correction Terms}
\label{sec:separation}

For a simplicial complex $K$ of dimension at least two, the main obstruction
to a deletion--contraction formula is liftability: after one fixes a new
matched pair $(\tau,\sigma)$, an acyclic matching on the remaining face poset
may become cyclic after the pair is inserted.  Since every finite simplicial
complex can be obtained by successively adjoining simplices whose boundaries
are already present, the following Top-Face Recursion gives a universal
one-step recursion for $\ME_K$, with the exact lift obstruction encoded by
$F(K,\sigma,\tau)$.

Equivalently, from the viewpoint of the Chari--Joswig complex, the recursion is
a weighted star decomposition with respect to the new matching vertices
$(\tau,\sigma)$, where $\tau\prec\sigma$.  An acyclic matching on $K'$ either
leaves $\sigma$ critical, contributing $z_d\ME_K$, or uses exactly one such
new vertex.  For a fixed facet $\tau$, the corresponding term counts those
matchings on $Q_{\sigma,\tau}$ that remain acyclic after adjoining the birth
pair $(\tau,\sigma)$; $F(K,\sigma,\tau)$ counts precisely the failures, i.e.
the matchings for which this birth step creates a closed gradient path.

\subsection{The Top-Face Recursion}
\label{subsec:top-face-recursion}

For any finite graded sub-poset $Q$ of a face poset, we use $\ME_Q$ for the
analogous critical-vector generating function over acyclic matchings on $Q$,
with $c_i^Q(M)$ denoting the number of unmatched rank-$i$ elements.  When the
ambient poset is clear from the context we drop the superscript and write
$c_i(M)$.

\begin{definition}[Liftable contraction polynomial]
\label{def:contracted-ME}
Let $K$ be a finite simplicial complex, let $\sigma$ be a
$d$-simplex with $\sigma\notin K$ and $\partial\sigma\subset K$, and set
\[
  K' = K\cup\{\sigma\}.
\]
For a facet $\tau\prec\sigma$, define the graded sub-poset
\[
  Q_{\sigma,\tau}:=P(K')\setminus\{\sigma,\tau\}.
\]
Call an acyclic matching $M'\in\cA(Q_{\sigma,\tau})$ \emph{liftable} if
$M'\cup\{(\tau,\sigma)\}\in\cA(P(K'))$, and write
\[
  \cA^{\mathrm{lift}}_{\sigma,\tau}
  :=\{\,M'\in\cA(Q_{\sigma,\tau}) : M'\cup\{(\tau,\sigma)\}\in\cA(P(K'))\,\}
\]
for the set of such matchings.  The \emph{liftable contraction polynomial} is
\[
  L_{\sigma,\tau}
  :=
  \sum_{M'\in\cA^{\mathrm{lift}}_{\sigma,\tau}}
  \prod_{i\geq 0} z_i^{c_i(M')}.
\]
Equivalently, $L_{\sigma,\tau}$ is the part of
$\ME_{Q_{\sigma,\tau}}$ contributed by those acyclic matchings on
$Q_{\sigma,\tau}$ whose lift, obtained by adjoining the matched pair
$(\tau,\sigma)$, remains acyclic on $P(K')$.  The fixed pair
$(\tau,\sigma)$ contributes no variable factor, since both $\tau$ and
$\sigma$ are matched rather than critical.
\end{definition}

\begin{theorem}[Top-Face Recursion]
\label{thm:top-face}
Let $K$ be a finite simplicial complex, and let $\sigma$ be a
$d$-simplex with $\sigma\notin K$ and $\partial\sigma\subset K$.
Set $K'=K\cup\{\sigma\}$.  Then
\[
  \ME_{K'}
  =
  z_d\cdot\ME_K
  +
  \sum_{\tau\prec\sigma}
  L_{\sigma,\tau}.
\]
\end{theorem}

\begin{proof}
We partition $\cA(P(K'))$ according to the status of the newly attached
simplex $\sigma$.

First suppose that $\sigma$ is critical.  Then no matched pair uses
$\sigma$.  Since $\sigma$ is maximal in $P(K')$, the remaining matching
is precisely an acyclic matching on $P(K)$.  The simplex $\sigma$
contributes one critical $d$-cell, hence this case contributes
\[
  z_d\cdot\ME_K.
\]

Now suppose that $\sigma$ is matched.  Since $\sigma$ is maximal, it can
only be matched with a facet $\tau\prec\sigma$.  After fixing the pair
$(\tau,\sigma)$, the remaining matching is an acyclic matching $M'$ on
\[
  Q_{\sigma,\tau}=P(K')\setminus\{\sigma,\tau\}.
\]
Conversely, such an $M'$ gives an acyclic matching on $P(K')$ with
$\sigma$ matched to $\tau$ precisely when the lift
\[
  M'\cup\{(\tau,\sigma)\}
\]
is acyclic on $P(K')$.  These are exactly the matchings counted by
$L_{\sigma,\tau}$.  The cases are disjoint and
exhaust all acyclic matchings on $P(K')$, proving the formula.
\end{proof}

\begin{definition}[Non-liftable correction]
\label{def:F}
In the notation of Definition~\ref{def:contracted-ME}, define
\[
  F(K,\sigma,\tau)
  :=
  \ME_{Q_{\sigma,\tau}}
  -
  L_{\sigma,\tau}.
\]
Thus $F(K,\sigma,\tau)$ is the generating function of acyclic matchings
on $Q_{\sigma,\tau}$ whose lift by the pair $(\tau,\sigma)$ is not
acyclic on $P(K')$.  Equivalently,
\[
  L_{\sigma,\tau}
  =
  \ME_{Q_{\sigma,\tau}}-F(K,\sigma,\tau).
\]
The correction term is not an auxiliary error term in an approximation; it is
the exact generating function of the lift obstruction.
\end{definition}

\subsection{The non-liftable correction term}
\label{subsec:correction-term}

\begin{definition}[$d$-incidence graph]
\label{def:top-incidence}
Let $K$ be a finite simplicial complex and fix $d\geq 1$.
The \emph{$d$-incidence graph} $\Gamma_d(K)$ is the bipartite graph
whose vertices are the $(d-1)$-simplices and the $d$-simplices of $K$,
with an edge $\rho-\eta$ whenever $\rho\prec\eta$.

In the setting where $K'=K\cup\{\sigma\}$, with $\sigma$ a
$d$-simplex and $\tau\prec\sigma$, define
\[
  S_{\sigma,\tau}
  :=
  \{\rho\prec\sigma:\rho\neq\tau\},
  \qquad
  C_{K,\tau}
  :=
  \{\eta\in K_d:\tau\prec\eta\}.
\]
Thus $S_{\sigma,\tau}$ consists of the other facets of the new simplex
$\sigma$, while $C_{K,\tau}$ consists of the old $d$-simplices of $K$
incident to the deleted facet $\tau$.
\end{definition}

\begin{lemma}[Incidence-separation criterion]
\label{thm:incidence-separation}
With the notation above,
\[
  F(K,\sigma,\tau)=0
\]
if and only if $S_{\sigma,\tau}$ and $C_{K,\tau}$ lie in distinct
connected components of the punctured incidence graph
\[
  \Gamma_d(K)\setminus\{\tau\}.
\]
Equivalently, the lift obstruction is present precisely when there is
an incidence path in $\Gamma_d(K)\setminus\{\tau\}$ from another facet
of $\sigma$ to an old $d$-simplex containing $\tau$.
\end{lemma}

\begin{proof}
Write $Q:=P(K')\setminus\{\sigma,\tau\}$; a non-liftable matching is an acyclic
matching $M'$ on $Q$ for which $M'\cup\{(\tau,\sigma)\}$ contains a closed 
gradient path.  We show such an $M'$ exists if and only if $S_{\sigma,\tau}$ and
$C_{K,\tau}$ are joined by a path in $\Gamma_d(K)\setminus\{\tau\}$.

\emph{Necessity.}
Let $M'$ be non-liftable and $\gamma$ a closed gradient path of
$M'\cup\{(\tau,\sigma)\}$.  As $M'$ is acyclic, $\gamma$ uses the new pair
$(\tau,\sigma)$.  After traversing it, $\gamma$ leaves $\sigma$ through a facet
$\rho\prec\sigma$ with $\rho\neq\tau$, so $\rho\in S_{\sigma,\tau}$; it then runs
inside $Q$, and since $\tau\notin Q$ it cannot meet $\tau$ until the closing
step, so immediately beforehand it passes through an old $d$-simplex $\eta$ with
$\tau\prec\eta$, i.e.\ $\eta\in C_{K,\tau}$.  Forgetting orientations and keeping
only the top-dimensional incidences, the part of $\gamma$ from $\rho$ to $\eta$
is an incidence path in $\Gamma_d(K)\setminus\{\tau\}$ from $S_{\sigma,\tau}$ to
$C_{K,\tau}$.

\emph{Sufficiency.}
Conversely, choose a shortest such path,
\[
  \pi:\ \rho_0-\eta_1-\rho_1-\eta_2-\cdots-\rho_{r-1}-\eta_r,
  \qquad \rho_0\in S_{\sigma,\tau},\ \ \eta_r\in C_{K,\tau}.
\]
Its vertices lie in $Q$, so
$M_\pi:=\{(\rho_0,\eta_1),(\rho_1,\eta_2),\ldots,(\rho_{r-1},\eta_r)\}$
is a matching on $Q$.  We claim $M_\pi\in\cA(Q)$.  Being shortest, $\pi$ is a
simple induced path: its vertices are distinct, and the facets among
$\rho_0,\ldots,\rho_{r-1}$ contained in $\eta_i$ are exactly its neighbours
$\rho_{i-1}$ and $\rho_i$ (a further incidence $\rho_j\prec\eta_i$ with
$j\notin\{i-1,i\}$ would give a chord shortening $\pi$).  In any closed gradient path
of $M_\pi$, every lower cell is a matched facet, hence one of
$\rho_0,\ldots,\rho_{r-1}$, and every upper cell is one of $\eta_1,\ldots,\eta_r$.
After the matched step $\rho_{i-1}\to\eta_i$, the only matched facet of $\eta_i$
available for descent is $\rho_i$, which ascends to $\eta_{i+1}$; thus a gradient path
can only advance forward along $\pi$.  Since $\pi$ is simple and $\eta_r$ has no
further matched facet, no gradient path closes up, so $M_\pi$ is acyclic.

Finally, adjoining $(\tau,\sigma)$ yields the closed gradient path, written in the
standard Forman order,
\[
  \tau,\ \sigma,\ \rho_0,\ \eta_1,\ \rho_1,\ \ldots,\ \rho_{r-1},\ \eta_r,\ \tau,
\]
where the last step uses $\tau\prec\eta_r$.  Hence
$M_\pi\cup\{(\tau,\sigma)\}$ is not acyclic, $M_\pi$ is non-liftable, and
$F(K,\sigma,\tau)\neq 0$.
\end{proof}

\begin{corollary}[Obstruction-free attachments]
\label{cor:obstruction-free}
If $C_{K,\tau}=\varnothing$, or more generally if
$S_{\sigma,\tau}$ and $C_{K,\tau}$ are separated in
$\Gamma_d(K)\setminus\{\tau\}$, then
\[
  L_{\sigma,\tau}
  =
  \ME_{P(K')\setminus\{\sigma,\tau\}}.
\]
In particular, this holds whenever $\dim K<d$.  For $d=1$, the
criterion reduces to the usual bridge condition for the new edge
$\sigma$.
\end{corollary}

\subsection{The incidence-forest regime and correction-free recursion}
\label{subsec:berge-acyclic-recursion}

Lemma~\ref{thm:incidence-separation} shows that non-liftability is caused by
closed codimension-one incidence patterns.  We now isolate the
\emph{incidence-forest} regime, where such patterns are absent and the
Top-Face Recursion becomes correction-free.  The key point is that this regime
is equivalent to the existence of a codimension-one leaf order.

\begin{definition}[Internal top incidence graph]
\label{def:internal-incidence-graph}
Let $K$ be a finite pure $d$-dimensional simplicial complex.  We call a
$(d-1)$-simplex \textit{internal} if it is contained in at least two
$d$-simplices of $K$, and \textit{boundary} otherwise.  The \textit{internal top
incidence graph}
$\Gamma_d^{\mathrm{int}}(K)$ is the subgraph of $\Gamma_d(K)$ induced by all
$d$-simplices and all internal $(d-1)$-simplices.
\end{definition}

The boundary facets dropped here are leaf vertices of $\Gamma_d(K)$, so deleting
them cannot create or destroy a cycle; hence $\Gamma_d(K)$ is a forest if and
only if $\Gamma_d^{\mathrm{int}}(K)$ is.  Equivalently,
$\Gamma_d^{\mathrm{int}}(K)$ is the incidence graph of the hypergraph whose
vertices are $d$-simplices and whose hyperedges are internal facets; this is also known as 
Berge-acyclicity~\cite{Berge1989}.

\begin{definition}[Codimension-one leaf order]
\label{def:codim-one-leaf-order}
Let $K$ be a finite pure $d$-dimensional simplicial complex.  An ordering
$\sigma_1,\ldots,\sigma_N$ of the $d$-simplices of $K$ is a
\textit{codimension-one leaf order} if, with
\[
  B_j:=K^{(d-1)}\cup\{\sigma_1,\ldots,\sigma_j\},
  \qquad B_0:=K^{(d-1)},
\]
each $\sigma_j$ has at most one facet contained in a $d$-simplex of $B_{j-1}$.
Here $K^{(d-1)}$ denotes the $(d-1)$-skeleton, and
$B_0\subset B_1\subset\cdots\subset B_N=K$ is the filtration of $K$ by the
attached $d$-simplices.
\end{definition}

\begin{proposition}[Existence of codimension-one leaf orders]
\label{thm:leaf-order-characterization}
Let $K$ be a finite pure $d$-dimensional simplicial complex, $d\ge 1$.  The
following are equivalent:
\begin{enumerate}
\setlength{\itemsep}{2pt}
\item[\textup{(i)}] the $d$-simplices of $K$ admit a codimension-one leaf order;
\item[\textup{(ii)}] the internal top incidence graph $\Gamma_d^{\mathrm{int}}(K)$
is a forest (equivalently, $\Gamma_d(K)$ is a forest);
\item[\textup{(iii)}] there is no closed codimension-one incidence chain
\[
  \eta_1-\rho_1-\eta_2-\rho_2-\cdots-\eta_k-\rho_k-\eta_1,
  \qquad k\ge 3,
\]
in which the $\eta_i$ are distinct $d$-simplices, the $\rho_i$ are distinct
internal $(d-1)$-simplices, and $\rho_i\prec\eta_i,\eta_{i+1}$ (indices modulo
$k$).
\end{enumerate}
\end{proposition}

\begin{proof}
\textup{(ii)}$\Leftrightarrow$\textup{(iii)}: a cycle in the bipartite graph
$\Gamma_d^{\mathrm{int}}(K)$ alternates between $d$-simplices and internal
$(d-1)$-simplices, which is precisely a closed chain of the stated form.  The
case $k=2$ cannot occur for distinct $d$-simplices in a simplicial complex,
since two distinct $d$-simplices cannot share two distinct facets.

\textup{(i)}$\Rightarrow$\textup{(iii)}: given a leaf order and such a chain, let
$\eta_m$ be the chain simplex added last.  Both of its chain-neighbours precede
it, so $\eta_m$ shares the two distinct facets $\rho_{m-1},\rho_m$ with the old
$d$-dimensional part, contradicting the leaf condition.

\textup{(ii)}$\Rightarrow$\textup{(i)}: induct on the number of $d$-simplices.
In a forest $\Gamma_d^{\mathrm{int}}(K)$ every internal $(d-1)$-vertex has degree
at least two, so any leaf of a nontrivial component is a $d$-simplex; hence some
$d$-simplex $\sigma$ is incident to at most one internal facet, i.e.\ shares at
most one facet with the other $d$-simplices.  Deleting $\sigma$ leaves
$\Gamma_d^{\mathrm{int}}$ a forest, so by induction the remaining $d$-simplices
admit a leaf order, and appending $\sigma$ extends it.
\end{proof}

\begin{theorem}[Correction-free recursion in the incidence-forest regime]
\label{prop:leaf-order}
Let $K$ be a finite pure $d$-dimensional simplicial complex whose internal top
incidence graph $\Gamma_d^{\mathrm{int}}(K)$ is a forest.  Choose a
codimension-one leaf order $\sigma_1,\ldots,\sigma_N$, whose existence is
guaranteed by Proposition~\ref{thm:leaf-order-characterization}.  Then, for every
$j$ and every facet $\tau\prec\sigma_j$,
\[
  F(B_{j-1},\sigma_j,\tau)=0.
\]
Consequently the Top-Face Recursion is correction-free along this order:
\[
  \ME_{B_j}
  =
  z_d\ME_{B_{j-1}}
  +
  \sum_{\tau\prec\sigma_j}
  \ME_{P(B_j)\setminus\{\sigma_j,\tau\}}
  \qquad (1\le j\le N).
\]
\end{theorem}

\begin{proof}
Fix $j$ and set $K_0:=B_{j-1}$.  By the definition of a codimension-one leaf
order, among the facets of $\sigma_j$ there is at most one facet, say
$\tau_0$, which is contained in an old $d$-simplex of $K_0$.  If
$\tau\neq\tau_0$, then $C_{K_0,\tau}=\varnothing$, and hence
$F(K_0,\sigma_j,\tau)=0$ by Corollary~\ref{cor:obstruction-free}.

It remains to consider the possible exceptional facet $\tau_0$.  For every
other facet $\rho\prec\sigma_j$ with $\rho\neq\tau_0$, the leaf-order
condition gives $C_{K_0,\rho}=\varnothing$.  Thus each such $\rho$ is isolated
from the old $d$-simplices in the incidence graph $\Gamma_d(K_0)$.  After
removing $\tau_0$, there is therefore no path in
$\Gamma_d(K_0)\setminus\{\tau_0\}$ from $S_{\sigma_j,\tau_0}$ to
$C_{K_0,\tau_0}$.  The incidence-separation criterion,
Lemma~\ref{thm:incidence-separation}, gives
$F(K_0,\sigma_j,\tau_0)=0$.

Substituting
$L_{\sigma_j,\tau}=\ME_{P(B_j)\setminus\{\sigma_j,\tau\}}$
into Theorem~\ref{thm:top-face} proves the displayed recursion.
\end{proof}

\begin{remark}[Relation with Forman acyclicity and shellability]
\label{rem:berge-vs-forman}
The incidence-forest condition should be distinguished from two nearby notions.
First, it is not Forman acyclicity: the latter constrains a particular matching
on the face poset, whereas forestness of $\Gamma_d^{\mathrm{int}}(K)$ is a
condition on the top-dimensional incidence structure of $K$ itself, before any
matching is chosen.  The two are linked only by projection: a closed gradient path
created in the lifting step has, after forgetting orientations and
lower-dimensional parts, a closed codimension-one incidence chain as its shadow,
so a forest incidence graph rules out exactly the top-level gradient cycles
measured by $F(K,\sigma,\tau)$, though not lower-dimensional closed gradient paths of
arbitrary matchings.

Second, it is not shellability: a shelling controls the intersection of each new
facet with the previous union (a pure $(d-1)$-subcomplex of its boundary), while
the present condition only sees $\Gamma_d^{\mathrm{int}}(K)$.  Neither implies
the other.
\end{remark}

\begin{corollary}[Stacked balls]
\label{cor:stacked-balls}
Let $K$ be a stacked $d$-ball, that is, a complex obtained from a single
$d$-simplex by repeatedly attaching a new $d$-simplex along a boundary
$(d-1)$-face, with a new apex vertex at each step.  Then $\Gamma_d(K)$ is a forest, equivalently the construction
order of its $d$-simplices is a codimension-one leaf order.  Consequently
every top-face attachment in this order is correction-free, and the Morse
ensemble polynomial of $K$ is obtained by iterating
\[
  \ME_{B_j}
  =
  z_d\ME_{B_{j-1}}
  +
  \sum_{\tau\prec\sigma_j}
  \ME_{P(B_j)\setminus\{\sigma_j,\tau\}}.
\]
In particular, for $d=2$ this applies to triangulated disks built by attaching
triangles successively along boundary edges.
\end{corollary}

\begin{proof}
In a stacked-ball construction, the first $d$-simplex has no old
$d$-dimensional neighbour.  Each later $d$-simplex is glued along exactly one
boundary facet of the previous complex, and all its other facets are new
boundary faces.  Hence, at the moment it is added, it has at most one facet
contained in an old $d$-simplex.  This is precisely the codimension-one leaf
order condition, so Proposition~\ref{thm:leaf-order-characterization} implies that
$\Gamma_d(K)$ is a forest.  The correction-free recursion follows from
Theorem~\ref{prop:leaf-order}.  Consequently stacked $d$-balls provide an
infinite family in every dimension for which the full Morse ensemble polynomial
is computable by a correction-free recursion.
\end{proof}

\begin{example}[A two-triangle stacked ball]
\label{ex:two-triangle-stacked-ball}
Let $K$ be the $2$-dimensional complex on vertices $\{0,1,2,3\}$ with
triangles
\[
  \sigma_1=\{0,1,2\},\qquad \sigma_2=\{0,1,3\}.
\]
Thus $K$ is the smallest non-degenerate stacked $2$-ball, with
$f(K)=(4,5,2)$.  Let $B_0$ be the one-skeleton of $K$, and set
\[
  B_1=B_0\cup\{\sigma_1\},
  \qquad
  B_2=B_1\cup\{\sigma_2\}=K.
\]
The first attachment has no old $2$-simplex.  The second attachment shares
only the edge $\{0,1\}$ with the old $2$-simplex $\sigma_1$.  Hence the order
$\sigma_1,\sigma_2$ is a codimension-one leaf order, and Corollary~\ref{cor:stacked-balls}
can be applied twice.  Explicitly,
\[
  \ME_{B_j}
  =
  z_2\ME_{B_{j-1}}
  +
  \sum_{\tau\prec\sigma_j}
  \ME_{P(B_j)\setminus\{\sigma_j,\tau\}},
  \qquad j=1,2,
\]
with no correction term in either step.  Direct enumeration, equivalently the
above two-step recursion, gives
\begin{align*}
  \ME_K
  ={}&
  32z_0
  +78z_0^2z_1
  +48z_0^3z_1^2
  +8z_0^4z_1^3 \\
  &+80z_0z_1z_2
  +116z_0^2z_1^2z_2
  +48z_0^3z_1^3z_2
  +6z_0^4z_1^4z_2 \\
  &+32z_0z_1^2z_2^2
  +32z_0^2z_1^3z_2^2
  +10z_0^3z_1^4z_2^2
  +z_0^4z_1^5z_2^2.
\end{align*}
The sum of the coefficients is $491$, agreeing with direct enumeration of
acyclic matchings on $P(K)$.  This example is the first instance in which the
correction-free recursion genuinely iterates through more than one
$2$-simplex.
\end{example}

\begin{example}[Iterating the recursion: $\Delta^2$ over $C_3$]
\label{ex:tfr-delta2}
Take $K=\partial\Delta^2=C_3$ and let $\sigma=\{0,1,2\}$ be the unique
$2$-simplex, so that $K'=\Delta^2$.  Since $\dim K<2$, the non-liftable
correction vanishes for every facet $\tau\prec\sigma$ by
Corollary~\ref{cor:obstruction-free}.  Hence
\[
  L_{\sigma,\tau}
  =
  \ME_{P(K')\setminus\{\sigma,\tau\}}.
\]
For each of the three edges $\tau$ of $\sigma$, the poset
$P(K')\setminus\{\sigma,\tau\}$ is the face poset of a path $P_3$.
Thus
\[
  \ME_{P_3}=3z_0+4z_0^2z_1+z_0^3z_1^2.
\]
Using
\[
  \ME_{C_3}(z_0,z_1)
  =
  9z_0z_1+6z_0^2z_1^2+z_0^3z_1^3,
\]
the Top-Face Recursion gives
\begin{align*}
  \ME_{\Delta^2}
  &=
  z_2\ME_{C_3}+3\ME_{P_3} \\
  &=
  9z_0
  +9z_0z_1z_2
  +12z_0^2z_1
  +6z_0^2z_1^2z_2
  +3z_0^3z_1^2
  +z_0^3z_1^3z_2.
\end{align*}
The sum of the coefficients is $40$, agreeing with direct enumeration
of acyclic matchings on $P(\Delta^2)$.
\end{example}

The preceding criterion detects whether the correction term vanishes.  When it
is nonzero, the bad matchings are governed by the same obstruction paths.
Indeed, if
\[
  \pi:\rho_0-\eta_1-\rho_1-\cdots-\rho_{r-1}-\eta_r
\]
is an incidence path in \(\Gamma_d(K)\setminus\{\tau\}\) from
\(S_{\sigma,\tau}\) to \(C_{K,\tau}\), then it determines the forced matching
\[
  M_\pi=\{(\rho_0,\eta_1),(\rho_1,\eta_2),\ldots,
  (\rho_{r-1},\eta_r)\}.
\]
A non-liftable matching contains the forced pairs associated to at least one
such obstruction path.  Thus \(F(K,\sigma,\tau)\) may be viewed as a finite
union, or equivalently an inclusion--exclusion, over obstruction paths.  The
first nonzero \(z_d\)-layer is controlled by the shortest ones:

\begin{proposition}[Leading obstruction layer]
\label{prop:leading-obstruction}
Suppose that \(S_{\sigma,\tau}\) and \(C_{K,\tau}\) lie in the same component of
\(\Gamma_d(K)\setminus\{\tau\}\), and set
\[
  \delta=
  \operatorname{dist}_{\Gamma_d(K)\setminus\{\tau\}}
  (S_{\sigma,\tau},C_{K,\tau})=2r-1.
\]
Then
\[
  \deg_{z_d} F(K,\sigma,\tau)=f_d(K)-r.
\]
Moreover, let \(N_{\mathrm{forc}}\) be the number of distinct forced matchings
\(M_\pi\) arising from shortest obstruction paths.  Then the coefficient of
\[
  \left(\prod_{i=0}^{d-2} z_i^{f_i(K)}\right)
  z_{d-1}^{f_{d-1}(K)-1-r}z_d^{f_d(K)-r}
\]
in \(F(K,\sigma,\tau)\) is at least \(N_{\mathrm{forc}}\).  If every non-liftable
matching contributing to this coefficient arises from a shortest obstruction
path, then it is exactly \(N_{\mathrm{forc}}\).
\end{proposition}

\begin{proof}
Every non-liftable matching contributing to \(F(K,\sigma,\tau)\) contains an
obstructing \(V\)-path.  Its top-dimensional part projects to an incidence path
in \(\Gamma_d(K)\setminus\{\tau\}\) from \(S_{\sigma,\tau}\) to
\(C_{K,\tau}\), hence has length at least \(2r-1\).  Therefore the matching uses
at least \(r\) matched \((d{-}1,d)\)-pairs.  Since each matched \(d\)-simplex lowers
the \(z_d\)-exponent by one, every non-liftable term has \(z_d\)-degree at most
\(f_d(K)-r\).

Conversely, every shortest obstruction path gives the forced matching
\(M_\pi\) constructed in Lemma~\ref{thm:incidence-separation}; its lift with
\((\tau,\sigma)\) contains the closed \(V\)-path displayed there.  Such a
matching uses exactly \(r\) matched \((d{-}1,d)\)-pairs, so the degree
\(f_d(K)-r\) is attained.  The shortest obstruction paths give
\(N_{\mathrm{forc}}\) distinct forced matchings contributing to the displayed
coefficient, hence the stated lower bound.  Under the additional hypothesis in
the last sentence of the statement, there are no further non-liftable matchings
contributing to this coefficient, so the lower bound is an equality.
\end{proof}

\begin{remark}[Analogy with Tutte deletion-contraction]
\label{rem:V-paths}
As in the bridge case (Remark~\ref{rem:non-bridge}), the Top-Face Recursion is
the ME analogue of Tutte deletion--contraction, with contraction performed in
the face poset: matching $\sigma$ with a facet $\tau$ removes only the pair
$(\tau,\sigma)$, leaving the other facets in the poset, so there is one
contraction-type term per facet of $\sigma$.  The new feature in dimension
$\ge 2$ is the lift condition in $L_{\sigma,\tau}$: $F(K,\sigma,\tau)$ is the
generating function of matchings on $P(K')\setminus\{\sigma,\tau\}$ for which
adjoining $(\tau,\sigma)$ creates a closed gradient path.  In dimension one, $F=0$
reduces to the bridge condition, recovering Theorem~\ref{thm:recursion}.
\end{remark}

\subsection{Pseudomanifold routing}
\label{subsec:pseudomanifold-routing}

The incidence-forest regime explains when the correction term vanishes.  At
the opposite end, pseudomanifold attachments typically have nonzero correction
terms, but the liftability test has a concrete dual-graph form.  We isolate it
as a deterministic routing criterion.

This is in the spirit of the use of discrete Morse theory on triangulated
manifolds and manifolds with boundary, especially in Benedetti's work on
boundary-critical discrete Morse functions and on the PL interpretation of
discrete Morse vectors~\cite{BenedettiBoundary,BenedettiSmoothing}.  The purpose
here is different: we do not construct manifold Morse functions, prove relative
Morse inequalities, or compare discrete and PL handle vectors.  The result below
is a local statement about the Top-Face Recursion: for a fixed matching on
$Q_{\sigma,\tau}$, the obstruction to adjoining the pair $(\tau,\sigma)$ is
detected by a routed path in the punctured dual graph.

Recall that a pure $d$-dimensional complex is a \emph{$d$-pseudomanifold}
if every $(d-1)$-simplex lies in exactly two $d$-simplices.  Its
\emph{dual graph} $H$ has the $d$-simplices as vertices and the
$(d-1)$-simplices as edges, each edge joining the two $d$-simplices that
contain it.  Let $K'$ be a $d$-pseudomanifold and let $\sigma$ be a
$d$-simplex.  Put $K=K'\setminus\{\sigma\}$ and write
$\rho_0,\ldots,\rho_d$ for the facets of $\sigma$.  In $K$, each
$\rho_i$ lies in a unique $d$-simplex $a_i$, namely its other coface in
$K'$.  Removing $\sigma$ from the dual graph leaves a graph $H_\sigma$ on
the $d$-simplices of $K$.  Each facet $\rho_i$ of $\sigma$, no longer shared
with a second $d$-simplex, becomes a \emph{boundary half-edge} of $H_\sigma$, a
dangling half-edge attached to $a_i$.

\begin{definition}[Gradient trajectory from a boundary facet]
\label{def:gradient-trajectory}
Let $M'$ be an acyclic matching on
$Q_{\sigma,\tau}=P(K')\setminus\{\sigma,\tau\}$, where $\tau=\rho_j$.  The
\emph{gradient trajectory from $\tau$} is the walk in the dual graph $H_\sigma$
that starts at the $d$-simplex $a_j$ and repeats the following step:
\begin{itemize}
\setlength{\itemsep}{1pt}
\item if $\eta$ is unmatched, stop;
\item if $\eta$ is matched in $M'$ to a facet $\rho$, leave $\eta$ through
$\rho$.  If $\rho$ is internal, move to the unique $d$-simplex on the other
side of $\rho$; if $\rho=\rho_k$ is a boundary half-edge, stop and say that
the trajectory \emph{exits at $\rho_k$}.
\end{itemize}
\end{definition}

The pseudomanifold condition makes the trajectory deterministic: each internal
$(d-1)$-simplex has exactly two cofaces, so after leaving a matched
$d$-simplex there is a unique next $d$-simplex.

\begin{proposition}[Trajectory criterion for liftability on pseudomanifolds]
\label{prop:forest-routing}
Let $K'$ be a $d$-pseudomanifold, let $\sigma$ be a $d$-simplex, let
$\tau=\rho_j$ be a facet of $\sigma$, and let $M'$ be an acyclic matching on
$Q_{\sigma,\tau}$.  Then $M'$ is non-liftable if and only if the gradient
trajectory of $M'$ from $\tau$ exits at a facet $\rho_k$ of $\sigma$ with
$k\neq j$.
\end{proposition}

\begin{proof}
We use the directed-cycle convention for the modified Hasse diagram: matched
cover relations are reversed, while unmatched cover relations keep their upward
orientation.  This is equivalent to the usual closed  gradient path condition,
up to cyclic reparametrisation of the closed path.

Since $M'$ is acyclic, any directed cycle of
$M'\cup\{(\tau,\sigma)\}$ must use the new reversed edge $\sigma\to\tau$.
After this edge, the cycle moves upward from $\tau$ to its unique old coface
$a_j$ in $K$.  From then on, the path is forced: at a matched $d$-simplex it
leaves through the matched facet, and the pseudomanifold condition gives a
unique next $d$-simplex across every internal $(d-1)$-face.  Thus such a
directed cycle exists precisely when the gradient trajectory starting at $a_j$
exits through a boundary half-edge $\rho_k$ with $k\neq j$.  The converse is
obtained by adjoining the final upward step $\rho_k\prec\sigma$ and then the
new reversed edge $\sigma\to\tau$.
\end{proof}

Thus, in the pseudomanifold case, liftability of a fixed matching no longer
requires a global directed-cycle search in the Hasse diagram: it is reduced to
following a single deterministic trajectory in the punctured dual graph,
although this does not by itself give a closed formula for $F(K,\sigma,\tau)$.
Geometrically, the matched top-dimensional part of $M'$ turns $H_\sigma$ into a
directed forest rooted at the unmatched $d$-simplices and the boundary
half-edges, and $M'$ is non-liftable exactly when the tree containing $a_j$ is
rooted at a boundary half-edge $\rho_k$ with $k\neq j$.

\begin{example}[The tetrahedral sphere]
\label{ex:forest-routing-bdy}
Let $K'=\partial\Delta^3$, the boundary of the tetrahedron, a
$2$-pseudomanifold, and let $\sigma$ be any of its $2$-simplices.  For each
facet $\tau\prec\sigma$ the poset $Q_{\sigma,\tau}$ carries $835$ acyclic
matchings, of which exactly $98$ are non-liftable.  Direct enumeration
confirms that these $98$ are precisely the matchings whose gradient
trajectory from $\tau$ exits at one of the other two edges of $\sigma$, as
predicted by Proposition~\ref{prop:forest-routing}.
\end{example}



\subsection{Perfect coefficients}
\label{subsec:perfect-coefficients}

We end this section by recording the coefficient of \(\ME_K\) that detects
perfect discrete Morse matchings.  Fix a field \(\Bbbk\), and write
\[
  z^{\beta_\Bbbk(K)}
  =
  \prod_i z_i^{\beta_i(K;\Bbbk)}.
\]
The \emph{perfect coefficient} of \(K\) over \(\Bbbk\) is
\[
  p_\Bbbk(K):=[z^{\beta_\Bbbk(K)}]\ME_K.
\]
It counts acyclic matchings whose critical vector is the Betti vector of
\(K\) over \(\Bbbk\).  Thus \(p_\Bbbk(K)>0\) if and only if \(K\) admits a
perfect discrete Morse matching over \(\Bbbk\).  For instance, the classical
shellability construction yields such matchings for pure shellable complexes
\cite{Chari2000}; equivalently, in the present language, their Betti
coefficient is nonzero.

The Top-Face Recursion projects immediately to this coefficient.

\begin{proposition}[Perfect coefficient recursion]
\label{prop:perfect-coefficient-recursion}
Let \(K'=K\cup\{\sigma\}\), where \(\sigma\) is a \(d\)-simplex and
\(\partial\sigma\subset K\).  Then
\[
  p_\Bbbk(K')
  =
  [z^{\beta_\Bbbk(K')-e_d}]\ME_K
  +
  \sum_{\tau\prec\sigma}
  [z^{\beta_\Bbbk(K')}]
  \bigl(\ME_{Q_{\sigma,\tau}}-F(K,\sigma,\tau)\bigr),
\]
where a coefficient with a negative exponent is interpreted as zero.
\end{proposition}

\begin{proof}
This is obtained by taking the coefficient of \(z^{\beta_\Bbbk(K')}\) in the
Top-Face Recursion of Theorem~\ref{thm:top-face}, using
\(L_{\sigma,\tau}=\ME_{Q_{\sigma,\tau}}-F(K,\sigma,\tau)\).
\end{proof}

In particular, the same correction term \(F(K,\sigma,\tau)\) that measures the
failure of naive contraction also controls the enumeration of perfect
matchings.  If the attachment of \(\sigma\) creates a new \(d\)-class, then
\(\beta_\Bbbk(K')=\beta_\Bbbk(K)+e_d\), and the first term contributes
\(p_\Bbbk(K)\).  If it kills a \((d-1)\)-class, then the perfect coefficient
is controlled entirely by the liftable contraction terms.

\section{The Independence ME Polynomial $\Phi(G)$ as a Graph Invariant}
\label{sec:independence}

\subsection{The invariant $\Phi(G)$}

The graph formula recalled in Section~\ref{sec:laplacian} gives the
one-dimensional calibration of the invariant, while
Section~\ref{sec:separation} develops the higher-dimensional correction theory.
The invariant $\Phi(G)$ will simultaneously refine the independence-polynomial
data of $\Ind(G)$ and the Laplacian-spectral data encoded by the graph-level
Morse ensemble polynomial $\ME_G$.
We now prove the second main result: a graph invariant which strictly refines
the graph-level Morse ensemble $\ME_G$ and distinguishes examples not separated
by $T_G$ or $I(G;t)$. The idea is to apply the ME polynomial to a
higher-dimensional simplicial complex canonically associated to a graph,
\textit{its independence complex}.

This places $\Phi(G)$ in the spirit of the large literature relating
independence complexes, edge ideals, and discrete Morse theory.  The complex
$\Ind(G)$ is the Stanley--Reisner complex of the edge ideal of $G$, and discrete
Morse methods have long been used to simplify cellular resolutions of monomial
ideals~\cite{BatziesWelker2002}; the topology of independence complexes is also
studied extensively via shellability, connectivity, and Cohen--Macaulay-type
properties~\cite{Jonsson2008,Engstrom2009,Woodroofe2011}.  The invariant
considered here is not a resolution, a homotopy-type computation, or a
shellability criterion.  It records the full critical-vector distribution over
all acyclic matchings on $\Ind(G)$; the separation results below therefore give
new enumerative Morse information rather than reprove known topological
properties of independence complexes.

\begin{definition}
\label{def:Phi}
For a graph $G=(V,E)$, the \emph{independence complex}
$\Ind(G)=\{I\subseteq V\mid I\text{ independent in }G\}$
is a simplicial complex of dimension $\alpha(G)-1$,
where $\alpha(G)$ is the independence number, the size of the largest independent set.
The \emph{independence polynomial}
$I(G;t)=\sum_{k\geq 0}i_k(G)\,t^k$
records the number $i_k(G)$ of independent sets of size $k$.

We define the \emph{independence ME polynomial}
\[
  \Phi(G):=\ME_{\Ind(G)}(z_0,z_1,\ldots,z_{\alpha(G)-1}),
\]
the Morse ensemble polynomial of the independence complex of $G$.
\end{definition}

The invariant $\Phi(G)$ can be understood as follows.
While $\ME_G$ encodes the discrete Morse structure on the vertex-edge
incidence of $G$ (a $1$-dimensional problem), $\Phi(G)$ encodes the
full Morse ensemble on the independence complex $\Ind(G)$, a
higher-dimensional complex whose topology reflects the independence
structure of $G$. Specifically:
\begin{itemize}
\item $I(G;t)=f(\Ind(G);t)$ reads only the $f$-vector of $\Ind(G)$
  (how many simplices in each dimension), while $\Phi(G)$ reads the
  full Morse structure---how many critical simplices of each dimension
  can arise in any acyclic matching.
\item $T_G$ reads only the cycle matroid (which subsets of edges are
  forests), while $\Phi(G)$ encodes the homotopy and combinatorial
  topology of $\Ind(G)$, which reflects the global independence
  structure of $G$ beyond forests.
\end{itemize}

\begin{theorem}
\label{thm:main}
There exist non-isomorphic graphs $G_1,G_2$ with
$T_{G_1}=T_{G_2}$ and $I(G_1;t)=I(G_2;t)$,
but $\Phi(G_1)\neq\Phi(G_2)$.
Thus $\Phi(G)$ is not determined by the Tutte polynomial together
with the independence polynomial.
\end{theorem}

\begin{proof}
Let $V=\{0,1,2,3,4,5\}$ and
\[
  E(G_1)=\{01,02,03,05,15,34\},
  \qquad
  E(G_2)=\{01,02,03,12,14,25\}.
\]

\begin{figure}[ht]
\centering
\begin{tikzpicture}[
  nd/.style={circle,draw,fill=black!10,inner sep=2pt,minimum size=18pt,font=\small},
  scale=1.05
]
\begin{scope}[xshift=0cm]
  \node[nd] (g1v0) at (0.8,0)  {$0$};
  \node[nd] (g1v1) at (2.4,0.9){$1$};
  \node[nd] (g1v2) at (2.4,-0.9){$2$};
  \node[nd] (g1v3) at (0.8,-1.5){$3$};
  \node[nd] (g1v4) at (2.4,-2.0){$4$};
  \node[nd] (g1v5) at (-0.8,0.9){$5$};
  \draw (g1v0)--(g1v1);
  \draw (g1v0)--(g1v2);
  \draw (g1v0)--(g1v3);
  \draw (g1v0)--(g1v5);
  \draw (g1v1)--(g1v5);
  \draw (g1v3)--(g1v4);
  \node[font=\small] at (0.8,-2.8) {$G_1$};
\end{scope}
\begin{scope}[xshift=5.5cm]
  \node[nd] (g2v0) at (0.0,0)  {$0$};
  \node[nd] (g2v1) at (1.5,1.0){$1$};
  \node[nd] (g2v2) at (1.5,-1.0){$2$};
  \node[nd] (g2v3) at (-1.4,0) {$3$};
  \node[nd] (g2v4) at (2.9,1.0){$4$};
  \node[nd] (g2v5) at (2.9,-1.0){$5$};
  \draw (g2v0)--(g2v1);
  \draw (g2v0)--(g2v2);
  \draw (g2v0)--(g2v3);
  \draw (g2v1)--(g2v2);
  \draw (g2v1)--(g2v4);
  \draw (g2v2)--(g2v5);
  \node[font=\small] at (0.8,-2.8) {$G_2$};
\end{scope}
\end{tikzpicture}
\caption{The graphs $G_1$ (left) and $G_2$ (right) witnessing
Theorem~\ref{thm:main}. Both have $6$ vertices, $6$ edges, the same
Tutte polynomial $T_{G_1}=T_{G_2}$, and the same independence polynomial
$I(G_1;t)=I(G_2;t)=1+6t+9t^2+4t^3$, yet $\Phi(G_1)\neq\Phi(G_2)$.}
\label{fig:main}
\end{figure}
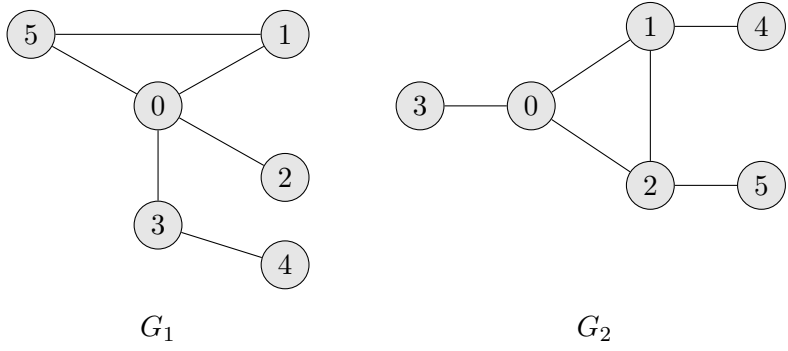

\emph{Non-isomorphism:} degree sequences
$(4,2,2,2,1,1)\neq(3,3,3,1,1,1)$.

\emph{Same Tutte polynomial:}
Direct computation by deletion-contraction (verified on all
$2^6=64$ edge subsets) gives
\begin{equation}
  T_{G_1}=T_{G_2}=x^5+x^4+x^3y.
\end{equation}
Both graphs are connected with $6$ vertices, $6$ edges, and $3$
spanning trees; the cycle matroids agree (each has a unique circuit
of size~$3$), and since the Tutte polynomial depends only on the
cycle matroid, the equality follows.

\emph{Same independence polynomial:}
Direct enumeration yields
$I(G_1;t)=I(G_2;t)=1+6t+9t^2+4t^3$,
so $\alpha(G_1)=\alpha(G_2)=3$ and both $\Ind(G_i)$ have
$f$-vector $(f_0,f_1,f_2)=(6,9,4)$.

\emph{Different $\Phi$:}
Although $\Ind(G_1)$ and $\Ind(G_2)$ share the same $f$-vector,
their degree sequences as $1$-skeleta are
$(4,4,3,3,3,1)$ and $(4,4,4,2,2,2)$ respectively, so
$\Ind(G_1)\not\cong\Ind(G_2)$ as simplicial complexes.
However, non-isomorphic complexes need not have distinct ME
polynomials (ME is an invariant but not a complete invariant),
so direct computation is required. Enumerating all acyclic
matchings on $P(\Ind(G_i))$ gives
\[
  [z_0]\,\Phi(G_1) = 270 \;\neq\; 324 = [z_0]\,\Phi(G_2),
\]
where $[z_0]$ denotes the coefficient of $z_0^1z_1^0z_2^0$
(matchings in which exactly one vertex is critical and all
edges and triangles are non-critical).
Hence $\Phi(G_1)\neq\Phi(G_2)$.

\medskip

We note a stronger observation. The independence complexes
$\Ind(G_1)$ and $\Ind(G_2)$ have the same $f$-vector $(6,9,4)$,
and both satisfy $[z_0]\Phi(G_i)\neq 0$:
\[
  [z_0]\Phi(G_1) = 270,\qquad [z_0]\Phi(G_2) = 324.
\]

Thus $\Phi(G)$ distinguishes these graphs \emph{despite} their
independence complexes sharing the same homotopy type.
\end{proof}

The equalities $T_{G_1}=T_{G_2}$ and $I(G_1;t)=I(G_2;t)$ encode
only the cycle matroid structure and the $f$-vector of $\Ind(G_i)$,
respectively. Neither captures the fine combinatorial structure
of $\Ind(G_i)$ required to separate these complexes.
The remarkable aspect of Theorem~\ref{thm:main} is not merely that
$\Phi(G)$ separates more graphs than $T_G$ or $I(G;t)$, but that
it does so even when the separated complexes share homotopy type.
The ME polynomial of the independence complex is therefore a
genuinely combinatorial invariant, sensitive to information beyond
homotopy type, f-vector, and cycle matroid.

The coefficient $[z_0]\Phi(G)$ counts collapsing matchings of
$\Ind(G)$; in particular, it is positive precisely when $\Ind(G)$ is
collapsible. This is not the main reason for introducing $\Phi$, but it
explains the witness in Theorem~\ref{thm:main}: both independence complexes
are collapsible, yet the collapse counts differ,
\[
  [z_0]\Phi(G_1)=270,\qquad [z_0]\Phi(G_2)=324.
\]
Thus $\Phi$ records collapse-level combinatorial information invisible
to $T_G$, $I(G;t)$, and homotopy type.

\subsection{Independence of $\Phi$ from Laplacian invariants}

The pair of graphs in Theorem~\ref{thm:main} satisfies
$T_{G_1}=T_{G_2}$ but does \emph{not} have
$\ME_{G_1}=\ME_{G_2}$ (the Laplacian spectra differ).
 We next
show that even the graph-level Morse ensemble polynomial $\ME_G$ does
not determine $\Phi(G)$.

\begin{theorem}[$\Phi$ separates Laplacian-cospectral graphs]
\label{thm:phi-cospectral}
There exist Laplacian-cospectral non-isomorphic graphs $G_1,G_2$
with $\ME_{G_1}=\ME_{G_2}$ but $\Phi(G_1)\neq\Phi(G_2)$.
In particular, $\Phi$ is not determined by $\ME_G$ even when
$\ME_G$ is the complete Laplacian spectral invariant.
\end{theorem}

\begin{proof}
Take the Laplacian-cospectral pair from
Example~\ref{ex:cospectral}: graphs $G_1,G_2$ on $6$ vertices
and $7$ edges with edge sets
$E(G_1)=\{01,02,03,05,14,23,45\}$ and
$E(G_2)=\{01,02,03,14,15,24,25\}$, both having Laplacian
eigenvalues $\{0,3-\sqrt 5, 2, 3, 3, 3+\sqrt 5\}$.
By the converse direction of
Theorem~\ref{cor:cospectral}, $\ME_{G_1}=\ME_{G_2}$.

Direct enumeration of acyclic matchings on the independence
complexes gives
\[
  |\cA(\Ind(G_1))| = 6212 \;\neq\; 15464 = |\cA(\Ind(G_2))|,
\]
so in particular $\Phi(G_1)\neq\Phi(G_2)$ as polynomials.
Furthermore, direct homology computation gives
\[
  \widetilde H_1(\Ind(G_1);\mathbb Z)\cong \mathbb Z,
\]
whereas $\Ind(G_2)$ is collapsible
($[z_0]\Phi(G_2)=144>0$), and hence contractible. Thus the two independence
complexes have different homotopy types. The equality $[z_0]\Phi(G_1)=0$
records the stronger collapse-level fact that $\Ind(G_1)$ is not collapsible.
\end{proof}

We next prove the converse structural statement: although $\ME_G$ does
not determine $\Phi(G)$, the independence $\ME$ polynomial $\Phi(G)$ does
determine the graph-level Morse ensemble polynomial $\ME_G$.

\begin{theorem}[$\Phi$ determines the graph ME polynomial]
\label{thm:phi-determines-me}
For every finite graph $G$, the polynomial $\Phi(G)$ determines
$\ME_G$. More precisely, $\Phi(G)$ determines the graph ME polynomial of the
complement graph $\overline{G}$, and hence determines the Laplacian spectrum
of $G$ and the polynomial $\ME_G$.
\end{theorem}

\begin{proof}
Let $K=\Ind(G)$ and let $H=K^{(1)}$ be its $1$-skeleton. Then
$H=\overline{G}$, the complement graph of $G$: the edges of $H$ are exactly
the independent two-element subsets of $G$, equivalently the non-edges of
$G$. Write $f_i=f_i(K)$ for the number of $i$-simplices of $K$, and let
$d=\dim K$. The vector $(f_0,\ldots,f_d)$ is visible
from $\Phi(G)$ as the exponent vector of the unique monomial of total degree
$|K|$, hence the monomial coming from the empty matching.

Consider the part of $\Phi(G)=\ME_K$ in which every simplex of
dimension at least $2$ is critical, namely the coefficient of
$z_2^{f_2}\cdots z_d^{f_d}$. A matching contributing to this
coefficient cannot contain any matched pair incident to a simplex of
dimension at least $2$. Therefore it is exactly an acyclic matching
on the vertex-edge part of the face poset, i.e. on the face poset of
$H=\overline{G}$. Conversely, any acyclic matching on $P(H)$ extends to
an acyclic matching on $P(K)$ by leaving all higher-dimensional
simplices unmatched, since all cover relations incident to higher
simplices remain oriented upward and cannot create a directed cycle.
Thus
\[
  \bigl[z_2^{f_2}\cdots z_d^{f_d}\bigr]\Phi(G)
  = \ME_{\overline{G}}(z_0,z_1).
\]

The graph Morse ensemble polynomial determines the Laplacian spectrum
of the graph by the Laplacian Formula (with the usual multiplicative
extension to disconnected graphs). Finally, the Laplacian spectra of
$G$ and its complement determine each other: if
$0=\lambda_1,\lambda_2,\ldots,\lambda_n$ are the Laplacian eigenvalues
of $G$, then the Laplacian eigenvalues of $\overline{G}$ are
\[
  0,\; n-\lambda_2,\; \ldots,\; n-\lambda_n.
\]
Consequently $\ME_{\overline{G}}$ determines the Laplacian spectrum of
$G$, and hence determines $\ME_G$.
\end{proof}

Combining Theorem~\ref{thm:phi-cospectral} with
Theorem~\ref{thm:phi-determines-me}, we obtain the strict hierarchy
\[
  \Phi(G) \quad\Longrightarrow\quad \ME_G,
  \qquad
  \ME_G \not\Longrightarrow \Phi(G).
\]
Thus the independence ME polynomial is a strict refinement of the
one-dimensional graph Morse ensemble invariant, rather than an
independent invariant in the opposite direction.

\begin{remark}[Recoverable invariants from $\Phi(G)$]
\label{rem:phi-recoverable}
The polynomial $\Phi(G)$ contains several standard graph-theoretic
quantities.  The empty matching in $\ME_{\Ind(G)}$ gives the monomial
\[
  \prod_i z_i^{f_i(\Ind(G))},
\]
so $\Phi(G)$ determines the full $f$-vector of $\Ind(G)$, and hence the
independence polynomial
\[
  I(G;t)=1+\sum_{i\geq 0} f_i(\Ind(G))t^{i+1}.
\]
In particular, it determines the independence number
\(\alpha(G)=\dim\Ind(G)+1\).

Moreover, by Theorem~\ref{thm:phi-determines-me}, $\Phi(G)$ determines
the graph-level Morse ensemble polynomial $\ME_G$, and hence the
Laplacian spectrum of $G$.  Thus $\Phi(G)$ refines both the
independence-polynomial data of $\Ind(G)$ and the Laplacian-spectral
data encoded by $\ME_G$.
\end{remark}

\begin{remark}[Matroid independence complexes]
Matroid independence complexes provide a natural class of examples for
the perfect-coefficient viewpoint.  If $M$ is a loopless matroid of rank
$r$, then its independence complex $\Ind(M)$ is pure shellable
\cite{Bjorner1992} and has the homotopy type of a wedge of
$(r-1)$-spheres.  Thus matroid independence complexes give natural test
cases for the coefficient
\[
  [z_0z_{r-1}^{\,\beta_{r-1}(\Ind(M))}]\ME_{\Ind(M)}.
\]

For example, $\Ind(U_{2,4})$ is the one-dimensional complex $K_4$.
Hence, by the Laplacian Formula,
\[
  \ME_{\Ind(U_{2,4})}
  =
  64z_0z_1^3
  +48z_0^2z_1^4
  +12z_0^3z_1^5
  +z_0^4z_1^6.
\]
Here $\beta_1(K_4)=3$, and
\[
  [z_0z_1^3]\ME_{K_4}=64
\]
counts the perfect matchings.

Similarly, $\Ind(U_{3,4})=\partial\Delta^3\simeq S^2$, and direct
enumeration gives
\[
  [z_0z_2]\ME_{\partial\Delta^3}=256.
\]
\end{remark}

\section{Open problems}
\label{sec:rigidity}

\textbf{Open problems.}
\begin{enumerate}
\item \emph{Closed forms for the non-liftable correction.}
\label{prob:nonliftable-matchings}
The correction $F(K,\sigma,\tau)$ is the main obstruction to a closed
deletion--contraction formula.  Its vanishing and leading layer are governed by
incidence geometry (Lemma~\ref{thm:incidence-separation} and
Proposition~\ref{prop:leading-obstruction}), and the obstruction-path viewpoint
expresses it as a finite inclusion--exclusion over forced-matching conditions.
Can this be converted into effective closed forms beyond the correction-free
incidence-forest regime, for instance for incidence graphs with one controlled
cycle, bounded treewidth, manifold-like pseudomanifolds, or the one-dimensional
non-bridge case?

\item \emph{Perfect Morse counts in higher dimensions.}
\label{prob:perfect-counts}
For graphs, $[z_0 z_1^{\beta_1}]\ME_G = n\tau(G)$ via the Matrix-Tree theorem
(Theorem~\ref{thm:perfect}). For the two-dimensional sphere $\partial\Delta^3$, the example
$[z_0z_2]\ME_{\partial\Delta^3}=256=4\cdot4\cdot\tau(K_4)$ suggests that
higher-dimensional perfect coefficients may be related to cellular spanning
tree enumerators. Duval--Klivans--Martin's simplicial Matrix-Tree
Theorem~\cite{DKM2009} provides the natural candidate tool, but the correct
formula likely involves compatible families of weighted cellular spanning trees
rather than a single top-dimensional tree enumerator.

\item \emph{Complexity of Morse ensemble computation.}
Any exact algorithm producing \(\ME_K\) in a representation from which the
minimum total degree can be read would solve the optimal discrete Morse
matching problem.  In this explicit-output sense, exact computation of
\(\ME_K\) is at least as difficult as optimal Morse matching.

A sharper formulation is open for coefficient extraction, for specializations
such as \(\ME_K(1,\ldots,1)\), and for the perfect coefficient
\(p_\Bbbk(K)\).  The Top-Face Recursion suggests that the difficulty is
concentrated in the non-liftable correction terms \(F(K,\sigma,\tau)\).
Can this obstruction be bounded or efficiently described for useful classes
of complexes?

Finally, does the same difficulty already occur on flag complexes?  Equivalently,
what is the complexity of computing
\[
  \Phi(G)=\ME_{\Ind(G)}
\]
from a graph \(G\)?

\item \emph{Recovery from $\Phi(G)$.}
\label{prob:phi-recover}
Theorem~\ref{thm:phi-determines-me} shows that $\Phi(G)$ determines $\ME_G$,
and hence the Laplacian spectrum of $G$. Which further graph parameters are
functions of $\Phi(G)$? For instance, is the chromatic polynomial $\chi(G;t)$
always recoverable from $\Phi(G)$? For which restricted graph classes, such as
forests, bipartite graphs, or planar graphs, does $\Phi(G)$ separate
non-isomorphic graphs?
\end{enumerate}

\end{document}